\documentclass[a4paper,12pt,oneside]{article}

\usepackage[top=3cm,left=2.5cm,right=2.5cm,bottom=3cm]{geometry}
\usepackage{relsize}
\usepackage{lineno}
\usepackage{hyperref}
\usepackage{amsmath,amsthm,pb-diagram,amssymb,comment}
\usepackage{amsfonts,graphicx,color}
\usepackage{enumitem, fancyhdr, dsfont}
\usepackage{thmtools,cleveref}
\usepackage[normalem]{ulem}
\usepackage{stmaryrd}
\usepackage{textcomp}
\usepackage{mathabx}
\usepackage{contour}
\usepackage{authblk}

\usepackage{tikz,float}

\hypersetup{
    colorlinks=true,
    linkcolor=dodger,
    filecolor=dodger,
    urlcolor=dodger,
    citecolor=dodger,
}

    \newcommand{\thzfc}{\mathrm{ZFC}}

    \newcommand{\Mg}{\mathbf{Mg}}
    \newcommand{\Cn}{\mathbf{Cn}}

    \newcommand{\Iwf}{\mathcal{I}}
    \newcommand{\Jwf}{\mathcal{J}}
    \newcommand{\Mwf}{\mathcal{M}}
    \newcommand{\Nwf}{\mathcal{N}}

    \newcommand{\bfrak}{\mathfrak{b}}
    \newcommand{\cfrak}{\mathfrak{c}}
    \newcommand{\dfrak}{\mathfrak{d}}

    \newcommand{\add}{\mbox{\rm add}}
    \newcommand{\cov}{\mbox{\rm cov}}
    \newcommand{\non}{\mbox{\rm non}}
    \newcommand{\cof}{\mbox{\rm cof}}

    \newcommand{\Bor}{\mathds{B}}
    \newcommand{\Cor}{\mathds{C}}
    \newcommand{\Dor}{\mathds{D}}
    \newcommand{\Eor}{\mathds{E}}

    \newcommand{\LOCor}{\mathds{LOC}}

    \newcommand{\Por}{\mathds{P}}
    
    \newcommand{\Qor}{\mathds{Q}}

    \newcommand{\Qnm}{\dot{\mathds{Q}}}

    \newcommand{\R}{\mathbb{R}}

    \newcommand{\cf}{\mbox{\rm cf}}

    \newcommand{\imp}{{\ \mbox{$\Rightarrow$} \ }}

    \newcommand{\sii}{{\ \mbox{$\Leftrightarrow$} \ }}

    \newcommand{\la}{\langle}
    \newcommand{\ra}{\rangle}

   \newcommand{\On}{\mathbf{On}}

\newcommand{\Dbf}{\mathbf{D}}

\newcommand{\Rbf}{\mathbf{R}}

\newcommand{\Cbf}{\mathbf{C}}
\newcommand{\Lc}{\mathbf{Lc}}

\newcommand{\Lb}{\mathbf{Lb}}
\newcommand{\Hcal}{\mathcal{H}}
\newcommand{\Scal}{\mathcal{S}}
\newcommand{\id}{\mathrm{id}}

\newcommand{\leqT}{\preceq_{\mathrm{T}}}
\newcommand{\eqT}{\cong_{\mathrm{T}}}

\newcommand{\COB}{\mathrm{COB}}

\newcommand{\LCU}{\mathrm{LCU}}
\newcommand{\EUB}{\mathrm{EUB}}

\newcommand{\baire}{\omega^\omega}

\contourlength{0.8pt}

\contourlength{0.8pt}

\definecolor{aguamarina}{cmyk}{0.85,0,0.33,0}
\definecolor{cafe}{cmyk}{0,0.81,1,0.60}
\definecolor{canela}{cmyk}{0.14,0.42,0.56,0}
\definecolor{darkgray}{cmyk}{0,0,0,0.50}
\definecolor{emerald}{cmyk}{0.91,0,0.88,0.12}
\definecolor{fresa}{cmyk}{0,1,0.50,0}
\definecolor{gold}{cmyk}{0,0.10,0.84,0}
\definecolor{lightgray}{cmyk}{0,0,0,0.30}
\definecolor{marron}{cmyk}{0,0.72,1,0.45}
\definecolor{melon}{cmyk}{0,0.29,0.84,0}
\definecolor{ladri}{cmyk}{0,0.77,0.87,0}
\definecolor{olive}{cmyk}{0.64,0,0.95,0.40}
\definecolor{orange}{cmyk}{0,0.42,1,0}
\definecolor{peach}{cmyk}{0,0.46,0.50,0}
\definecolor{pink}{cmyk}{0,0.10,0.10,0}
\definecolor{orange}{cmyk}{0,0.42,1,0}
\definecolor{pine}{cmyk}{0.92,0,0.59,0.25}
\definecolor{purple}{cmyk}{0.45,0.86,0,0}
\definecolor{violet}{cmyk}{0.07,0.90,0,0.34}

\newcommand{\ccafe}[1]{{\color{cafe}#1}}
\newcommand{\ccanela}[1]{{\color{canela}#1}}

\newcommand{\cmarron}[1]{{\color{marron}#1}}
\newcommand{\cmelon}[1]{{\color{melon}#1}}
\newcommand{\cladri}[1]{{\color{ladri}#1}}
\newcommand{\colive}[1]{{\color{olive}#1}}
\newcommand{\corange}[1]{{\color{orange}#1}}

\newcommand{\blue}[1]{{\color{blue}#1}}
\newcommand{\cyan}[1]{{\color{cyan}#1}}
\newcommand{\green}[1]{{\color{green}#1}}

\DeclareSymbolFont{extraup}{U}{zavm}{m}{n}
\DeclareMathSymbol{\varheart}{\mathalpha}{extraup}{86}
\DeclareMathSymbol{\vardiamond}{\mathalpha}{extraup}{87}

\definecolor{dodger}{rgb}{0.0,0.5,1.0}

\newcommand{\azul}[1]{{\color{blue}#1}}

\definecolor{amber}{rgb}{1.0,0.49,0.0}

\definecolor{ogreen}{RGB}{107,142,35}

\title{Forcing constellations of Cicho\'n's diagram by using the Tukey order}

\begin{document}

\author[1]{Miguel A.~Cardona}
\author[2]{Diego A.~Mej\'ia}

\date{}

\affil[1]{Institute of Discrete Mathematics and Geometry, TU Wien.}
\affil[2]{Faculty of Sciences, Shizuoka University.}

\newcommand{\Addresses}{{
  \bigskip
  \footnotesize
  
  Institute of Discrete Mathematics and Geometry\\
  TU Wien\\
  Wiedner Hauptstrasse 8–10/104, A-1040 Wien\\
  AUSTRIA\\
  E--mail address: \texttt{miguel.montoya@tuwien.ac.at}
  \bigskip

  Creative Science Course (Mathematics), Faculty of Science\\
  Shizuoka University\\
  836 Ohya, Suruga-ku, Shizuoka 422-8529\\
  JAPAN\\
  E--mail address: \texttt{diego.mejia@shizuoka.ac.jp}
}}

\maketitle

\makeatletter
\def\@roman#1{\romannumeral #1}
\makeatother

\newcounter{enuAlph}
\renewcommand{\theenuAlph}{\Alph{enuAlph}}

\numberwithin{equation}{section}
\renewcommand{\theequation}{\thesection.\arabic{equation}}


\theoremstyle{plain}
  \newtheorem{theorem}[equation]{Theorem}
  \newtheorem{corollary}[equation]{Corollary}
  \newtheorem{lemma}[equation]{Lemma}
  \newtheorem{mainlemma}[equation]{Main Lemma}
  \newtheorem{fact}[equation]{Fact}
  \newtheorem{prop}[equation]{Proposition}
  \newtheorem{claim}[equation]{Claim}
  \newtheorem{question}[equation]{Question}
  \newtheorem{problem}[equation]{Problem}
  \newtheorem{conjecture}[equation]{Conjecture}
  \newtheorem*{theorem*}{Theorem}
  \newtheorem*{mainthm*}{Main Theorem}
  \newtheorem{teorema}[enuAlph]{Theorem}
  \newtheorem*{corollary*}{Corollary}
\theoremstyle{definition}
  \newtheorem{definition}[equation]{Definition}
  \newtheorem{example}[equation]{Example}
  \newtheorem{remark}[equation]{Remark}
  \newtheorem{notation}[equation]{Notation}
  \newtheorem{context}[equation]{Context}
  \newtheorem{observation}[equation]{Observation}
  \newtheorem{assumption}[equation]{Assumption}
  \newtheorem*{definition*}{Definition}
  \newtheorem*{acknowledgements*}{Acknowledgements}

\def\sectionautorefname{Section}
\def\subsectionautorefname{Subsection}


\begin{abstract}
We use known finite support iteration techniques to present various examples of models where several cardinal characteristics of Cicho\'n's diagram are pairwise different. We show some simple examples forcing the left-hand side of Cicho\'n's diagram, and present the technique of restriction to models to force Cicho\'n's maximum (original from Goldstern, Kellner, Shelah and the second author). We focus on how the values forced in all the constellations are obtained via the Tukey order.
\end{abstract}

\makeatother

\section*{Introduction}\label{sec:intro}

Let $\Iwf$ be an ideal of subsets of $X$ such that $\{x\}\in \Iwf$ for all $x\in X$. We define \emph{cardinal characteristics associated with $\Iwf$} by:

\begin{description}
\item[Additivity of $\Iwf$:] $\add(\Iwf)=\min\{|\Jwf|:\,\Jwf\subseteq\Iwf,\,\bigcup\Jwf\notin\Iwf\}$. 

\item[Covering of $\Iwf$:] $\cov(\Iwf)=\min\{|\Jwf|:\,\Jwf\subseteq\Iwf,\,\bigcup\Jwf=X\}$.

\item[Uniformity of $\Iwf$:] $\non(\Iwf)=\min\{|A|:\,A\subseteq X,\,A\notin\Iwf\}$.

\item[Cofinality of $\Iwf$:] $\cof(\Iwf)=\min\{|\Jwf|:\,\Jwf\subseteq\Iwf,\ \forall\, A\in\Iwf\ \exists\, B\in \Jwf\colon A\subseteq B\}$.
\end{description}

\begin{figure}[H]
  \begin{center}
    \includegraphics[scale=1.0]{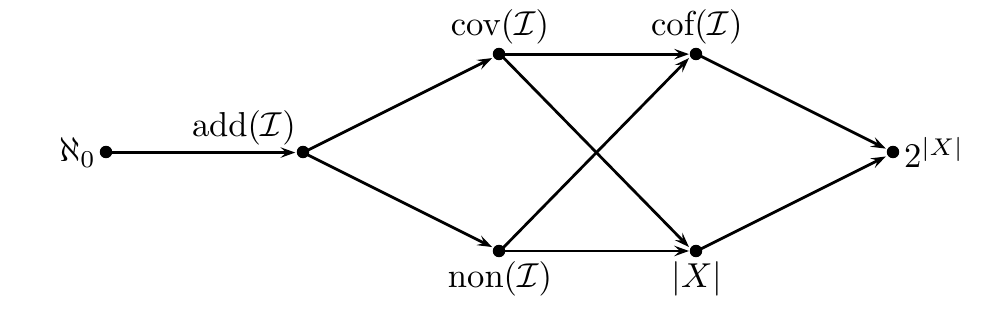}
    \caption{Diagram of the cardinal characteristics associated with $\Iwf$. An arrow  $\mathfrak x\rightarrow\mathfrak y$ means that (provably in ZFC) 
    $\mathfrak x\le\mathfrak y$.}
    \label{diag:idealI}
  \end{center}
\end{figure}

\autoref{diag:idealI} shows the natural inequalities between the cardinal characteristics associated with  $\Iwf$. These cardinals have been studied intensively for $\Mwf$ and $\Nwf$ (see e.g.~\cite{BJ,blass}), which denote the $\sigma$-ideal first category subsets of $\R$ and the $\sigma$-ideal of Lebesgue null subsets of $\R$, respectively. We denote, as usual, $\cfrak:=2^{\aleph_0}=|\R|$, and recall that $\aleph_1$ is the smallest uncountable cardinal. 

For $f,g\in\omega^\omega$ we write 
\[\text{$f\leq^* g$ (which is read \emph{$f$ is dominated by $g$}) iff $\exists\, m\ \forall\, n\geq m\colon f(n)\leq g(n)$.}\] 
In addition, we define
  \begin{itemize}
    \item[{}] The \emph{bounding number} $\bfrak=\min\{|F|:\, F\subseteq \omega^\omega\text{ and }\neg\exists\, y\in \omega^\omega\ \forall\, x\in F\colon x\leq^* y\}$, and
    \item[{}] the \emph{dominating number} $\dfrak=\min\{|D|:\, D\subseteq \omega^\omega\text{ and }\forall\, x\in \omega^\omega\ \exists\, y\in D\colon x\leq^* y\}$.
\end{itemize}  

The relationship between these cardinals is best illustrated by \emph{Cicho\'n's diagram}~(see \autoref{FigCichon}), which is one of the most important diagrams in set theory of the reals and has been a relevant object of study since the decade of the 1980's. It is well-known that this diagram is \emph{complete} in
the sense that no other inequality can be proved between two cardinal characteristics there.
See e.g.~\cite{BJ} for a complete survey about this diagram and its completeness.

\begin{figure}[ht]
\begin{center}
  \includegraphics[scale=1.3]{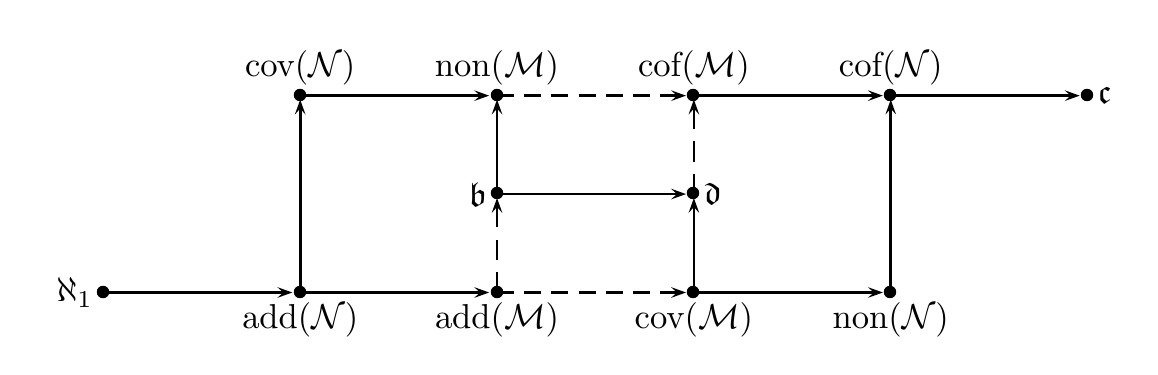}
  \caption{Cicho\'n's diagram. The arrows mean $\leq$ and dotted arrows represent
  $\add(\Mwf)=\min\{\bfrak,\cov(\Mwf)\}$ and $\cof(\Mwf)=\max\{\dfrak,\non(\Mwf)\}$.}
  \label{FigCichon}
\end{center}
\end{figure}
In the context of this diagram, a natural question arises:
\begin{quote}
    Is it consistent that all the cardinals in~\autoref{FigCichon} (with the exception of the dependent values $\add(\Mwf)$ and $\cof(\Mwf)$) are pairwise different?
\end{quote}
It turns out that the answer to this question is positive and was proved by Goldstern, Kellner and Shelah \cite{GKScicmax}, who used four strongly compact cardinals to obtain the consistency of Cicho\'n's diagram divided into 10 different values, situation known as \emph{Cicho\'n's maximum}. In this same direction. This was improved by Brendle and the authors~\cite{BCM} who used only three strongly compact cardinals; finally, Goldstern, Kellner, Shelah and the second author~\cite{GKMS} proved that no large cardinals are needed for the consistency of Cicho\'n's maximum. 

The previously cited work occurs in the context of finite support (FS) iterations of ccc posets. In fact, when calculating the values of the cardinals in Cicho\'n's diagram in generic extensions, Tukey connections appear implicitly. This appears a bit more explicitly in \cite{GKScicmax,GKMS} with the notions of $\COB$ (Cone of bounds) and $\LCU$ (linear cofinally unbounded), but still the full power of the Tukey connections remained unexplored.

To complement this last part, this work summarizes some of the techniques required to force Cicho\'n's maximum, but making the role of the Tukey order very explicit. This allows to reformulate all technical results and main theorems in a very beautiful and concise way.


\section{Relational systems and cardinal characteristics}\label{sec:relsys}

Many cardinal characteristics of the continuum and their relations can be represented by relational systems as follows. This presentation is based on~\cite{Vojtas,BartInv,blass}.

\begin{definition}\label{def:relsys}
We say that $\Rbf=\la X, Y, \sqsubset\ra$ is a \textit{relational system} if it consists of two non-empty sets $X$ and $Y$ and a relation $\sqsubset$.
\begin{enumerate}[label=(\arabic*)]
    \item A set $F\subseteq X$ is \emph{$\Rbf$-bounded} if $\exists\, y\in Y\ \forall\, x\in F\colon x \sqsubset y$. 
    \item A set $E\subseteq Y$ is \emph{$\Rbf$-dominating} if $\forall\, x\in X\ \exists\, y\in E\colon x \sqsubset y$. 
\end{enumerate}

We associate two cardinal characteristics with this relational system $\Rbf$: 
\begin{itemize}
    \item[{}] $\bfrak(\Rbf):=\min\{|F|:\, F\subseteq X  \text{ is }\Rbf\text{-unbounded}\}$ the \emph{unbounding number of $\Rbf$}, and
    
    \item[{}] $\dfrak(\Rbf):=\min\{|D|:\, D\subseteq Y \text{ is } \Rbf\text{-dominating}\}$ the \emph{dominating number of $\Rbf$}.
\end{itemize}
\end{definition}

A very representative general example of relational systems is given by directed preorders.

\begin{definition}\label{examSdir}
We say that $\la S,\leq_S\ra$ is a \emph{directed preorder} if it is a preorder (i.e.\ $\leq_S$ is a reflexive and transitive relation on $S$) such that 
\[\forall\, x, y\in S\ \exists\, z\in S\colon x\leq_S z\text{ and }y\leq_S z.\] 
A directed preorder $\la S,\leq_S\ra$ is seen as the relational system $S=\la S, S,\leq_S\ra$, and their associated cardinal characteristics are denoted by $\bfrak(S)$ and $\dfrak(S)$. The cardinal $\dfrak(S)$ is actually the \emph{cofinality of $S$}, typically denoted by $\cof(S)$ or $\cf(S)$.
\end{definition}

\begin{fact}\label{basicdir}
If a directed preorder $S$ has no maximum element then $\bfrak(S)$ is infinite and regular, and $\bfrak(S)\leq\cf(\dfrak(S))\leq\dfrak(S)\leq|S|$. Even more, if $L$ is a linear order without maximum then $\bfrak(L)=\dfrak(L)=\cof(L)$.
\end{fact}

The following list of examples are relevant for the main results of this paper.

\begin{example}\label{exm:Baire}
Consider $\omega^\omega=\la\omega^\omega,\leq^*\ra$, 
which is a directed preorder. The cardinal characteristics $\bfrak:=\bfrak(\omega^\omega)$ and $\dfrak:=\dfrak(\omega^\omega)$ are the well-known \emph{bounding number} and \emph{dominating number}, respectively.
\end{example}

\begin{example}\label{exm:Iwf}
For any ideal $\Iwf$ on $X$, we consider the following relational systems. 
\begin{enumerate}[label=(\arabic*)]
    \item $\Iwf:=\la\Iwf,\subseteq\ra$ is a directed partial order. Note that $\bfrak(\Iwf)=\add(\Iwf)$ and $\dfrak(\Iwf)=\cof(\Iwf)$.
    
    \item $\Cbf_\Iwf:=\la X,\Iwf,\in\ra$. Note that $\bfrak(\Cbf_\Iwf)=\non(\Iwf)$ and $\dfrak(\Cbf_\Iwf)=\cov(\Iwf)$.
\end{enumerate}
\end{example}

\begin{example}\label{ex:ideal<theta}
Let $\theta$ be an infinite cardinal and $X$ a set of size $\geq\theta$. Then $[X]^{<\theta}$ is an ideal. We look at its associated cardinal characteristics. 

Its additivity and uniformity numbers are easy to determine:
\[\add([X]^{<\theta})=\cf(\theta) \text{ and } \non([X]^{<\theta})=\theta.\]

For the covering number, we obtain
    \[\cov([X]^{<\theta})=\left\{
    \begin{array}{ll}
        |X| & \text{if $|X|>\theta$,} \\
        \cf(\theta) & \text{if $|X|=\theta$.}
    \end{array}\right.\]
Therefore $\cov([X]^{<\theta})=|X|$ whenever $\theta$ is regular, which is our case of interest.

The cofinality number is more interesting. Under Shelah's Strong Hypothesis\footnote{The failure of this hypothesis requires large cardinals.} it follows that
\[\cof([X]^{<\theta})=\left\{
    \begin{array}{ll}
        |X| & \text{if $\cf(|X|)\geq\theta$,} \\
        |X|^+ & \text{otherwise.}
    \end{array}\right.\]
In ZFC, we have $\cof([X]^{<\theta})=|X|$ whenever $|X|^{<\theta}=|X|$, which is our case of interest.
\end{example}

Inequalities between cardinal characteristics associated with relational systems can be determined by the dual of a relational system and also via Tukey connections, which we introduce below.

\begin{definition}\label{def:dual}
If $\Rbf=\la X,Y,\sqsubset\ra$ is a relational system, then its \emph{dual relational system} is defined by $\Rbf^\perp:=\la Y,X,\sqsubset^\perp\ra$ where $y \sqsubset^\perp x$ if $\neg(x \sqsubset y)$.
\end{definition}

\begin{fact}\label{fct:dual}
Let $\Rbf=\la X,Y,\sqsubset\ra$ be a relational system.
\begin{enumerate}[label=(\alph*)]
    \item $(\Rbf^\perp)^\perp=\Rbf$.
    \item The notions of $\Rbf^\perp$-dominating set and $\Rbf$-unbounded set are equivalent.
    \item The notions of $\Rbf^\perp$-unbounded set and $\Rbf$-dominating set are equivalent.
    \item $\dfrak(\Rbf^\perp)=\bfrak(\Rbf)$ and $\bfrak(\Rbf^\perp)=\dfrak(\Rbf)$.
\end{enumerate}
\end{fact}

\begin{definition}\label{def:Tukey}
Let $\Rbf=\la X,Y,\sqsubset\ra$ and $\Rbf'=\la X',Y',\sqsubset'\ra$ be relational systems. We say that $(\Psi_-,\Psi_+)\colon\Rbf\to\Rbf'$ is a \emph{Tukey connection from $\Rbf$ into $\Rbf'$} if 
 $\Psi_-\colon X\to X'$ and $\Psi_+\colon Y'\to Y$ are functions such that  \[\forall\, x\in X\ \forall\, y'\in Y'\colon \Psi_1(x) \sqsubset' y' \Rightarrow x \sqsubset \Psi_2(y').\]
The \emph{Tukey order} between relational systems is defined by
$\Rbf\leqT\Rbf'$ iff there is a Tukey connection from $\Rbf$ into $\Rbf'$. \emph{Tukey equivalence} is defined by $\Rbf\eqT\Rbf'$ iff $\Rbf\leqT\Rbf'$ and $\Rbf'\leqT\Rbf$
\end{definition}

\begin{fact}\label{fct:Tukey}
Assume that $\Rbf=\la X,Y,\sqsubset\ra$ and $\Rbf'=\la X',Y',\sqsubset'\ra$ are relational systems and that $(\Psi_-,\Psi_+)\colon \Rbf\to\Rbf'$ is a Tukey connection.
\begin{enumerate}[label=(\alph*)]
    \item If $D'\subseteq Y'$ is $\Rbf'$-dominating, then $\Psi_+[D']$ is $\Rbf$-dominating.
    \item $(\Psi_+,\Psi_-)\colon (\Rbf')^\perp\to\Rbf^\perp$ is a Tukey connection.
    \item If $E\subseteq X$ is $\Rbf$-unbounded then $\Psi_-[E]$ is $\Rbf'$-unbounded.
\end{enumerate}
\end{fact}

\begin{corollary}\label{cor:Tukeyval}
\begin{enumerate}[label=(\alph*)]
    \item $\Rbf\leqT\Rbf'$ implies $(\Rbf')^\perp\leqT\Rbf^\perp$.
    \item $\Rbf\leqT\Rbf'$ implies $\bfrak(\Rbf')\leq\bfrak(\Rbf)$ and $\dfrak(\Rbf)\leq\dfrak(\Rbf')$.
    \item $\Rbf\eqT\Rbf'$ implies $\bfrak(\Rbf')=\bfrak(\Rbf)$ and $\dfrak(\Rbf)=\dfrak(\Rbf')$.
\end{enumerate}
\end{corollary}

\begin{example}\label{ex:trivialTukey}
The diagram in \autoref{diag:idealI} can be expressed in terms of the Tukey order since $\Cbf_\Iwf\leqT\Iwf$ and $\Cbf_\Iwf^\perp\leqT\Iwf$.
\end{example}

\begin{example}\label{ex:tukeysmall}
If $\theta'\leq\theta$ are infinite cardinals, and $\theta\leq|X|\leq|X'|$, then $\Cbf_{[X]^{<\theta}}\leqT \Cbf_{[X']^{<\theta'}}$. On the other hand, for any regular cardinal $\mu$, $\Cbf_{[\mu]^{<\mu}}\eqT[\mu]^{<\mu}\eqT\mu$, so $\add([\mu]^{<\mu})=\cof([\mu]^{<\mu})=\mu$. As a consequence:
\end{example}

\begin{fact}\label{impEUB}
Assume that $\theta\leq\lambda$ are infinite cardinals. Then, for any regular $\mu\in[\theta,\lambda]$, $\mu\leqT\Cbf_{[\lambda]^{<\theta}}$.
\end{fact}

In fact, the inequalities in Cicho\'n's diagram (\autoref{FigCichon}) are obtained via the Tukey connections illustrated in \autoref{fig:cichontukey}.

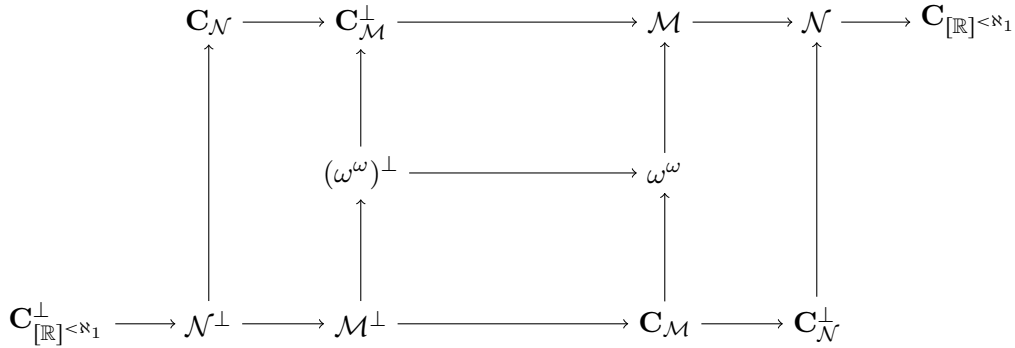
\begin{figure}[ht]
\centering
\begin{tikzpicture}
\small{
\node (aleph1) at (-1,3) {$\Cbf_{[\R]^{<\aleph_1}}^\perp$};
\node (addn) at (1,3){$\Nwf^\perp$};
\node (covn) at (1,7){$\Cbf_\Nwf$};
\node (nonn) at (9,3) {$\Cbf_\Nwf^\perp$} ;
\node (cfn) at (9,7) {$\Nwf$} ;
\node (addm) at (3,3) {$\Mwf^\perp$} ;
\node (covm) at (7,3) {$\Cbf_\Mwf$} ;
\node (nonm) at (3,7) {$\Cbf_\Mwf^\perp$} ;
\node (cfm) at (7,7) {$\Mwf$} ;
\node (b) at (3,5) {$(\baire)^\perp$};
\node (d) at (7,5) {$\baire$};
\node (c) at (11,7) {$\Cbf_{[\R]^{<\aleph_1}}$};
\draw (aleph1) edge[->] (addn)
      (addn) edge[->] (covn)
      (covn) edge [->] (nonm)
      (nonm)edge [->] (cfm)
      (cfm)edge [->] (cfn)
      (cfn) edge[->] (c);
\draw
   (addn) edge [->]  (addm)
   (addm) edge [->]  (covm)
   (covm) edge [->]  (nonn)
   (nonn) edge [->]  (cfn);
\draw (addm) edge [->] (b)
      (b)  edge [->] (nonm);
\draw (covm) edge [->] (d)
      (d)  edge[->] (cfm);
\draw (b) edge [->] (d);
}
\end{tikzpicture}
\caption{Cichon's diagram via Tukey connections. Any arrow represents a Tukey connection in the given direction.}\label{fig:cichontukey}
\end{figure}

In this paper, when we force a value of a cardinal characteristic via ccc posets, we actually force Tukey connections with relational systems of the form $\Cbf_{[\lambda]^{<\theta}}$ and $[\lambda]^{<\theta}$ for some cardinals $\theta\leq\lambda$ with $\theta$ uncountable regular. For instance, if $\Rbf$ is a relational system and we force $\Rbf\eqT\Cbf_{[\lambda]^{<\theta}}$, then we obtain $\bfrak(\Rbf)=\non([\lambda]^{<\theta})=\theta$ and $\dfrak(\Rbf)=\cov([\lambda]^{<\theta})=\lambda$. In the case when $\lambda^{<\theta}=\lambda$, we obtain the same values when forcing $\Rbf\eqT[\lambda]^{<\theta}$, also because of the following result.

\begin{lemma}\label{lem:thetaid}
If $\theta$ is a regular cardinal and $|X|^{<\theta}=|X|$, then $\Cbf_{[X]^{<\theta}}\eqT[X]^{<\theta}$.
\end{lemma}
\begin{proof}
The relation $\leqT$ is immediate from \autoref{ex:trivialTukey}. For the converse, since $Z:=[X]^{<\theta}$ has the same size as $X$, we get $\Cbf_{[Z]^{<\theta}}\eqT\Cbf_{[X]^{<\theta}}$ and $[Z]^{<\theta}\eqT [X]^{<\theta}$ by using a bijection from $X$ into $Z$, so it is enough to show that $[X]^{<\theta}\leqT\Cbf_{[Z]^{<\theta}}$. The Tukey connection is given by the identity map from $[X]^{<\theta}$ into $Z$, and by the map $\Psi_+\colon [Z]^{<\theta}\to [X]^{<\theta}$ defined by $\Psi_+(A):=\bigcup A$.
\end{proof}

Motivated by the previous explanation, we look at characterizations of the Tukey order between $\Cbf_{[X]^{<\theta}}$ and other relational systems.

\begin{lemma}\label{lem:TukeyCtheta}
Let $\theta$ be an infinite cardinal, $I$ a set of size $\geq\theta$ and let $\Rbf=\la X,Y,\sqsubset\ra$ be a relational system. Then: 
\begin{enumerate}[label=(\alph*)]
    \item If $|X|\geq\theta$, then $\Rbf\leqT\Cbf_{[X]^{<\theta}}$ iff $\forall\, A \in[X]^{<\theta}\ \exists\, y_A\in Y\ \forall\, x\in A\colon x\sqsubset y_A$, i.e.\ any subset of $X$ of size ${<}\theta$ is $\Rbf$-bounded. 
    
    In this case, when $\theta$ is regular, $\theta\leq\bfrak(\Rbf)$ and $\dfrak(\Rbf)\leq|X|$.
    
    \item $\Cbf_{[I]^{<\theta}}\leqT\Rbf$ iff $\exists\,\la x_i:\, i\in I\ra\subseteq X\, \forall\, y\in Y\colon |\{i\in I:\, x_i\sqsubset y\}|<\theta$.
    
    In this case, when $\theta$ is regular, $\bfrak(\Rbf)\leq\theta$ and $|I|\leq\dfrak(\Rbf)$.
\end{enumerate}
\end{lemma}
\begin{proof}
\noindent(a): 
The implication from right to left is immediate by using the maps $x\mapsto x$ (identity on $X$) and $A\mapsto y_A$ as a Tukey connection. For the converse, assume $\Rbf\leqT\Cbf_{[X]^{<\theta}}$, i.e., there is a Tukey connection $(F,G)\colon\Rbf\to\Cbf_{[X]^{<\theta}}$. For $A\in[X]^{<\theta}$, $y_A:=G(F[A])$ is as desired.
The latter part uses \autoref{ex:ideal<theta}.\smallskip

\noindent(b): The implication from right to left follows by using the maps $i\mapsto x_i$ and $y\mapsto \{i\in I:\, x_i\sqsubset y\}$. To see the converse, assume that $(F,G)\colon\Cbf_{[I]^{<\theta}}\to\Rbf$ is a Tukey connection. For $i\in I$, let $x_i:=F(i)$, so $\{i\in I:\, x_i\sqsubset y\}\subseteq G(y)$ and $|G(y)|<\theta$.
\end{proof}

When forcing constellations of Cicho\'n's diagram via ccc posets, we look at simpler characterizations of the relational systems of its cardinals by coding with reals.

\begin{definition}\label{def:defrel}
We say that $\Rbf=\la X, Y, \sqsubset\ra$ is a \textit{definable relational system of the reals} if both $X$ and $Y$ are non-empty and analytic in Polish spaces $Z$ and $W$, respectively, and $\sqsubset$ is analytic in $Z\times W$.
\end{definition}

\begin{remark}\label{rem:defrel}
In the previous definition indicates that any definable relational system is Tukey equivalent to a relational system of the form $\la\omega^\omega,\omega^\omega,\sqsubset\ra$ for some analytic relation $\sqsubset$ on $\baire$. Indeed, if $\Rbf$ is as in \autoref{def:defrel}, then Tukey connections are obtained by some Borel isomorphism from $\omega^\omega$ onto $Z$.
\end{remark}

To characterize the relational systems of Cicho\'n's diagram, we use relational systems with better definitions.

\begin{definition}\label{DefPolish}
We say that $\Rbf=\langle X,Y,\sqsubset\rangle$ is a \textit{Polish relational system (Prs)} if the following is satisfied:
\begin{enumerate}[label=(\roman*)]
\item $X$ is a perfect Polish space,
\item $Y$ is a non-empty analytic subspace of some Polish space $Z$ and
\item\label{Fsigma} $\sqsubset\cap(X\times Z)=\bigcup_{n<\omega}\sqsubset_{n}$ where $\la\sqsubset_{n}\ra_{n<\omega}$ is some increasing sequence of closed subsets of $X\times Z$ such that $(\sqsubset_{n})^{y}=\{x\in X:\, x\sqsubset_{n}y \}$ is closed nowhere dense for any $n<\omega$ and $y\in Y$.
\end{enumerate}
\end{definition}

By~\ref{Fsigma}, we obtain:

\begin{fact}\label{Prsfact}
If $\Rbf$ is a Prs then $\la X,\Mwf(X),\in\ra\leqT\Rbf$. Therefore, $\bfrak(\Rbf)\leq \non(\Mwf)$ and $\cov(\Mwf)\leq\dfrak(\Rbf)$.
\end{fact}

\begin{example}\label{ExmPrs}
The following are Prs that describe the cardinal characteristics of Cicho\'n's diagram.
  \begin{enumerate}[label=(\arabic*)]
     \item Define the relational system $\Mg:=\la2^\omega,\Xi,\in^{\bullet}\ra$ where 
     \[\Xi := \{f \colon 2^{<\omega}\to2^{<\omega}:\, \forall\, s \in 2^{<\omega}\colon s \subseteq f(s)\}\] 
     and $x\in^{\bullet} f \textrm{\ iff\ } |\{s\in 2^{<\omega}:\, x \supseteq f(s)\}|<\aleph_0$. This is a Prs and $\Mg\eqT\Cbf_\Mwf$. Hence $\bfrak(\Mg)=\non(\Mwf)$ and $\dfrak(\Mg)=\cov(\Mwf)$.
     
     \item  The relational system $\baire:=\la\omega^\omega,\omega^\omega,\leq^*\ra$ is already Polish. 

     \item  Define $\Omega_n:=\{a\in [2^{<\omega}]^{<\aleph_0}:\, \Lb_2(\bigcup_{s\in a}[s])\leq 2^{-n}\}$ (endowed with the discrete topology) where $\Lb_2$ is the Lebesgue measure on $2^\omega$. Put $\Omega:=\prod_{n<\omega}\Omega_n$ with the product topology, which is a perfect Polish space. For every $x\in \Omega$ denote $N_{x}^{*}:=\bigcap_{n<\omega}\bigcup_{s\in x(n)}[s]$, which is clearly a Borel null set in $2^{\omega}$.

     Define the Prs $\Cn:=\la \Omega, 2^\omega, \sqsubset\ra$ where $x\sqsubset z$ iff $z\notin N_{x}^{*}$. Recall that any null set in $2^\omega$ is a subset of $N_{x}^{*}$ for some $x\in \Omega$, so $\Cn\eqT\Cbf_\Nwf^\perp$. Hence, $\bfrak(\Cn)=\cov(\Nwf)$ and $\dfrak(\Cn)=\non(\Nwf)$.

     \item  For each $k<\omega$ let $\id^k\colon \omega\to\omega$ such that $\id^k(i)=i^k$ for all $i<\omega$, and $\Hcal:=\{\id^{k+1}:\, k<\omega\}$. Let $\Lc^*:=\la\omega^\omega, \Scal(\omega, \Hcal), \in^*\ra$ be the Polish relational system where \[\Scal(\omega, \Hcal):=\{\varphi\colon \omega\to[\omega]^{<\aleph_0}:\, \exists\, h\in\Hcal\ \forall\, i<\omega\colon |\varphi(i)|\leq h(i)\}\]
     and $x\in\varphi$ iff $x(i)\in\varphi(i)$ for all but finitely many $i$.
     As consequence of~\cite{BartInv}, $\Lc^*\eqT\Nwf$, so $\bfrak(\Lc^*)=\add(\Nwf)$ and $\dfrak(\Lc^*)=\cof(\Nwf)$.     
  \end{enumerate}
\end{example}

To conclude this section, we review products of relational systems.

\begin{definition}\label{def:prodrel}
  Let $\Rbf=\la X,Y,\sqsubset\ra$ and $\Rbf'=\la X',Y',\sqsubset'\ra$ be relational systems.
  Define the relational system $\Rbf\times\Rbf':=\la X\times X', Y\times Y', \sqsubset_\times\ra$ by
  \[(x,x')\sqsubset_\times (y,y') \sii x\sqsubset y\text{ and }x'\sqsubset'y'.\]
\end{definition}

\begin{fact}\label{fct:prodrel}
For relational systems $\Rbf$ and $\Rbf'$:
  \begin{enumerate}[label=(\alph*)]
    \item $\Rbf\leqT \Rbf\times\Rbf'$ and $\Rbf'\leqT \Rbf\times\Rbf'$.
    \item $\bfrak(\Rbf\times\Rbf')=\min\{\bfrak(\Rbf),\bfrak(\Rbf')\}$ and $\max\{\dfrak(\Rbf),\dfrak(\Rbf')\}\leq \dfrak(\Rbf\times\Rbf')\leq \dfrak(\Rbf)\cdot\dfrak(\Rbf')$.
    \item If $S$ and $S'$ are directed preorders, then so is $S\times S'$.
  \end{enumerate}
\end{fact}

In \autoref{sec:cmax} we use relational systems of the form $\Lambda:=\prod_{i<n}\nu_i$ for limit ordinals $\nu_i$ and $n<\omega$. Note that $\bfrak(\Lambda)=\min\{\cf(\nu_i):\, i<n\}$ and  $\dfrak(\Lambda)=\max\{\cf(\nu_i):\, i<n\}$.

\section{Forcing and Tukey connections}\label{sec:forcing}

We present general results illustrating the effect of FS iterations of ccc posets on the cardinal characteristics associated with a definable relational system of the reals. More concretely, if $\Rbf$ is such a relational system, we show how to force statements of the form $\Rbf\leqT\Cbf_{[I]^{<\theta}}$ and $\Cbf_{[I]^{<\theta}}\leqT\Rbf$.

We start by looking at special types of generic reals.

\begin{definition}\label{def:domreal}
Let $\Rbf=\la X,Y,\sqsubset\ra$ be a relational system and let $M$ be a set (commonly a model).
\begin{enumerate}[label=(\arabic*)]
    \item Say that $y\in Y$ is \emph{$\Rbf$-dominating over  $M$} if $\forall\, x \in  X\cap M\colon x \sqsubset y$.
    \item Say that $x$ is \emph{$\Rbf$-unbounded over  $M$} if it is $\Rbf^\perp$-dominating over $M$, that is, $\forall\, y \in  Y\cap M\colon \neg(x \sqsubset y)$.
\end{enumerate}
\end{definition}

\begin{example}\label{exm:Suslinccc}
The following are examples of very typical Suslin ccc forcing notions and the type of dominating (or unbounded) reals they add over the ground model. For precise definitions, see e.g.~\cite{BJ}.
\begin{enumerate}[label=(\arabic*)]
    \item Cohen reals are precisely the $\Cbf_\Mwf$-unbounded reals, which are precisely the $\Mg$-unbounded reals in the context of $2^\omega$. We denote Cohen forcing by $\Cor$.
    \item Random reals are precisely the $\Cbf_\Nwf$-unbounded reals, which are precisely the $\Cn^\perp$-dominating reals. We denote random forcing by $\Bor$.
    \item The eventually different real forcing $\Eor$ adds a $\Cbf_\Mwf$-dominating real in $\omega^\omega$, which can be transformed into a $\Mg$-dominating real (in $2^\omega$).
    \item Hechler forcing $\Dor$ adds an $\omega^\omega$-dominating real (usually called \emph{dominating real)}.
    \item $\LOCor$ adds an $\Lc^*$-dominating real, which also adds an $\Nwf$-dominating real.
\end{enumerate}
\end{example}

\begin{definition}
Let $\nu$ be an ordinal. An iteration $\la\Por_\xi,\Qnm_\xi:\, \xi<\nu\ra$ has \emph{finite support} (FS for short) if 
\begin{enumerate}[label=(\roman*)]
    \item $\xi<\delta\leq\nu\Rightarrow\Por_\xi\lessdot\Por_\delta$,
    \item $\Por_{\xi+1}=\Por_\xi\ast\Qnm_\xi$, and
    \item $\Por_\delta=\bigcup_{\xi<\delta}\Por_\xi$ for all limit $\delta\leq\nu$.
\end{enumerate}
We usually denote $V_\xi:=V^{\Por_\xi}$ for all $\xi\leq\nu$.
\end{definition}

It is important to have a reasonably good picture of a FS iteration before plunging into technical facts, see \autoref{FSI}.

\begin{figure}[ht]
\begin{center}
  \includegraphics[scale=1]{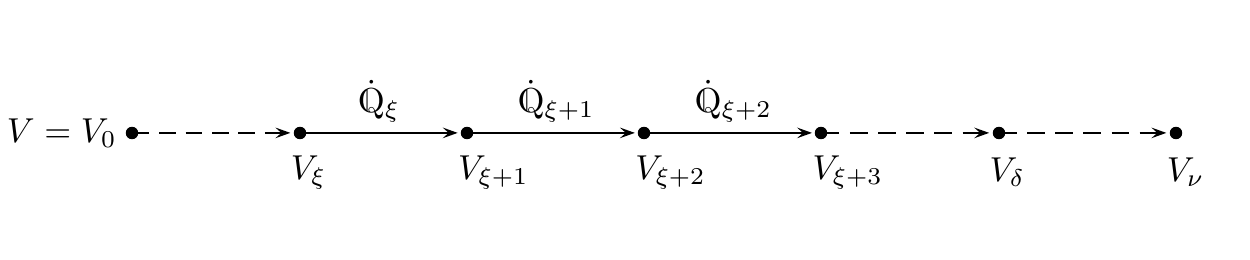}
  \caption{FS iteration of length $\nu$.}
  \label{FSI}
\end{center}
\end{figure}

\begin{remark}
For limit $\delta\leq\nu$, $V_\delta\neq\bigcup_{\xi<\delta}V_\xi$ in general.
\end{remark}

Below, we state some well-known facts (most of them without proofs) for FS iterations of forcing notions.
The following lemma states that, in FS 
iterations of certain forcing notions (e.g.\ ccc forcing notions), no new reals are added at limit stages of uncountable cofinality, a result which will be used often in forthcoming results.

\begin{lemma}\label{basicFS}
Let $\theta$ be a regular uncountable cardinal. If $\la\Por_\xi,\Qnm_\xi:\, \xi<\nu\ra$ is a FS iteration of $\theta$-cc posets, i.e.\ $\Vdash_{\Por_\xi}\Qnm_\xi$ is $\theta$-cc for all $\xi<\nu$, then $\Por_\nu$ is $\theta$-cc.

If, in addition, $\cf(\nu)\geq\theta$, then $\omega^\omega\cap V_\nu=\omega^\omega\cap\bigcup_{\xi<\nu}V_{\xi}$.
\end{lemma}

FS iterations add Cohen reals, which is sometimes considered as a limitation of the method.

\begin{lemma}\label{addingCohen}
Assume that $\Por_\nu=\la\Por_\xi,\Qnm_\xi:\,  \xi<\nu\ra$ is a FS iteration of non-trivial posets. 
If $\xi<\nu$ and $\nu$ is limit, then $\Por_\nu$ adds a Cohen real over $V_\xi$.
\end{lemma}

\begin{corollary}\label{Cohenlimit}
Let $\nu$ be a limit ordinal of uncountable cofinality and let $\Por_\nu=\la\Por_\xi,\Qnm_\xi:\,  \xi<\nu\ra$ be a FS iteration of non-trivial $\cf(\nu)$-posets. Then $\Por_\nu$ forces $\nu\leqT\Mg$. In particular, $\Por_\nu$ forces $\non(\Mwf)\leq\cf(\nu)\leq\cov(\Mwf)$. 
\end{corollary}
\begin{proof}
Since $L\eqT\nu$ for any cofinal subset $L$ of $\nu$, it is enough to show that $L^*\leqT\Cbf_\Mwf$ where $L^*$ is the set of limit ordinals smaller than $\nu$. For each $i\in L^*$ let $c_i\in V_{i+\omega}$ be a Cohen real over $V_i$ (which exists by \autoref{addingCohen}). Then the Tukey connection is given by the maps $i\mapsto c_i$ and $B\mapsto j_B$, the latter defined by: whenever $B$ is a Borel meager set of reals, $j_B\in L^*$ is chosen such that $B$ is coded in $V_{j_B}$ (which exists by \autoref{basicFS}).
\end{proof}

One starting point to force a statement of the form $\Cbf_{[I]^{<\theta}}\leqT\Rbf$ is the following result.

\begin{fact}\label{Cohenit}
Let $\Cor_I$ be the poset that adds Cohen reals indexed by $I$. If $I$ is uncountable then $\Cor_I$ forces $\Cbf_{[I]^{<\aleph_1}}\leqT\Mg$.
\end{fact}
\begin{proof}
Apply \autoref{lem:TukeyCtheta}~(b) to the sequence $\la c_i:\, i\in I\ra$ of Cohen reals added by $\Cor_I$.
\end{proof}

The following results illustrates the effect of adding cofinally many $\Rbf$-dominating reals along a FS iteration.

\begin{lemma}\label{thm:fullgen}
Let $\Rbf$ be a definable relational system of the reals, and let $\nu$ be a limit ordinal of uncountable cofinality.
If $\Por_\nu=\la\Por_\xi,\Qnm_\xi:\,  \xi<\nu\ra$ is a FS iteration of $\cf(\nu)$-cc posets that adds $\Rbf$-dominating reals cofinally often, then $\Por_\nu$ forces $\Rbf\leqT\nu$. 

In addition, if $\Rbf$ is a Prs and all iterands are non-trivial, then $\Por_\nu$ forces $\Rbf\eqT\Mg\eqT\nu$.
In particular, $\Por_\nu$ forces $\bfrak(\Rbf)=\dfrak(\Rbf)=\non(\Mwf)=\cov(\Mwf) =\cf(\nu)$.
\end{lemma}
\begin{proof}
Let $L$ be the set of $\xi<\nu$ such that $\Qnm_\xi$ add an $\Rbf$-dominating real over $V_\xi$. Since $L$ is cofinal in $\nu$, $L\eqT\nu$. We show that, in $V_\nu$, $\Rbf\leqT L$. Consider the maps $F\colon X^{V_\nu}\to L$ such that $x\in V_{F(x)}$, and $G\colon L\to Y^{V_\nu}$ such that $G(\xi)$ is $\Rbf$-dominating over $V_\xi$. Clearly, $(F,G)$ is a Tukey connection.

The second part follows from \autoref{Prsfact} and \autoref{Cohenlimit}.
\end{proof}

In most of the cases, we force a Tukey connection between a definable relational system of the reals and some relational system $R$ fixed in the ground model. To calculate the cardinal characteristics in the extension, we need to know when $\bfrak(R)$ and $\dfrak(R)$ stay the same in generic extensions.

\begin{lemma}\label{cccgenext0}
Let $\theta>\aleph_0$ be a regular cardinal and let $R=\la A,B,\sqsubset\ra$ be a relational system. 
\begin{enumerate}[label=(\alph*)]
    \item If $V\models \dfrak(R)\geq\theta$ then, in any $\theta$-cc generic extension of $V$, $\dfrak(R)=\dfrak(R)^V$. 
    \item If $V\models \bfrak(R)\geq\theta$ then, in any $\theta$-cc generic extension of $V$, $\bfrak(R)=\bfrak(R)^V$.
\end{enumerate}
Here, $R$ is considered as the same object in both $V$ and in the generic extension (not an interpretation).
\end{lemma}
\begin{proof}
We show (a) (note that (b) follows by (a) applied to $R^\perp$).
In $V$, assume that $\lambda:=\dfrak(R)\geq\theta$ and that $D\subseteq B$ is an $R$-dominating family of size $\lambda$. Let $W$ be a $\theta$-cc generic extension of $V$. In $W$, it is clear that $D$ is $R$-dominating, so $\dfrak(R)^W\leq\lambda$. Now assume, in $W$, that $E\subseteq B$ has size ${<}\lambda$. Since $W$ is a $\theta$-cc generic extension of $V$ and $\lambda\geq\theta$, we can find $E'\in V$ of size ${<}\lambda$ such that $E\subseteq E'\subseteq B$. In $V$, $|E'|<\lambda=\dfrak(R)$, so $E'$ is not $R$-dominating, hence there is some $x\in X$ which is $R$-unbounded over $E'$. It is clear that, in $W$, $x$ is $R$-unbounded over $E$. This concludes that $\dfrak(R)^W\geq\lambda$.
\end{proof}

In our applications, $R$ will be a directed set like $[X]^{<\theta}\cap V$, or a relational system of the form $\Cbf_{[X]^{<\theta}\cap V}$. In terms of Tukey equivalence, this will be the same as looking at $[X]^{<\theta}$ and  $\Cbf_{[X]^{<\theta}}$, respectively, in the generic extension.

\begin{lemma}\label{cccgenext}
Let $\theta>\aleph_0$ be a regular cardinal and assume $|X|\geq\theta$. Then, in any $\theta$-cc generic extension, $[X]^{<\theta}\eqT[X]^{<\theta}\cap V$ and $\Cbf_{[X]^{<\theta}}\eqT\Cbf_{[X]^{<\theta}\cap V}$. Moreover, $\add([X]^{<\theta})=\add([X]^{<\theta})^V$, likewise for the other cardinal characteristics associated with the ideal $[X]^{<\theta}$.
\end{lemma}
\begin{proof}
This follows because, in any $\theta$-cc generic extension of $V$, any member of $A\in[X]^{<\theta}$ is contained in some member of $[X]^{<\theta}\cap V$. The ``moreover" is a consequence of \autoref{cccgenext0} applied to $[X]^{<\theta}\cap V$ and $\Cbf_{[X]^{<\theta}\cap V}$.
\end{proof}

The following result is a general criteria to force statements of the form $\Rbf\leqT\Cbf_{[X]^{<\theta}}$.

\begin{theorem}\label{itsmallsets}
Let $\Rbf=\la X,Y,\sqsubset\ra$ be a definable relational system of the reals, $\theta$ an uncountable regular cardinal, and let  $\Por_\nu=\la\Por_\xi,\Qnm_\xi:\,  \xi<\nu\ra$ be a FS iteration of $\theta$-cc posets with $\cf(\nu)\geq\theta$. Assume that, for all $\xi<\nu$ and any $A\in[X]^{<\theta}\cap V_\xi$, there is some $\eta\geq\xi$ such that $\Qnm_\eta$ adds an $\Rbf$-dominating real over $A$. Then $\Por_\nu$ forces $\Rbf\leqT\Cbf_{[X]^{<\theta}}\eqT\Cbf_{[X]^{<\theta}\cap V}\leqT[X]^{<\theta}\eqT[X]^{<\theta}$, in particular, $\theta\leq\bfrak(\Rbf)$ and $\dfrak(\Rbf)\leq|X|=\cfrak$.
\end{theorem}
\begin{proof}
In $V_\nu$: for $A\in[X]^{<\theta}$ we have that $A\in V_\xi$ for some $\xi<\nu$, so it is $\Rbf$-bounded by the hypothesis. Hence $\Rbf\leqT\Cbf_{[X]^{<\theta}}$ by \autoref{lem:TukeyCtheta}. The rest follows by \autoref{ex:trivialTukey} and \autoref{cccgenext}.
\end{proof}

\begin{remark}\label{remCOB}
In connection with the theorem above, \cite{GKScicmax} defines the following property:
\begin{description}
\item[$\COB(\Rbf,\Por,\lambda,\vartheta)$] states that there are a directed preorder $S$ of size $\nu$ and with $\bfrak(S)\geq\lambda$, and a sequence $\la\dot y_s:\, s\in S\ra$ of $\Por$-names of members of $Y$ such that,
for every $\Por$-name $\dot x$ of a member of $X$, there exists an $s_{\dot x}\in S$ such that, for all $t \geq_S s_{\dot x}$, $\Vdash \dot x \sqsubset \dot y_t$.  
\end{description}
This property implies that there exist a directed preorder $S$ in the ground model such that 
$\lambda\leq\bfrak(S)$, $\dfrak(S)\leq\vartheta$ and $\Vdash_\Por\Rbf\leqT S$. Even more, equivalence holds when $\Por$ is $\lambda$-cc and $\lambda$ is uncountable regular.
\end{remark}

We conclude this section with general results to force statements of the form $\Cbf_{[I]^{<\mu}}\leqT\Rbf$. For this purpose, we restrict to Polish relational systems and use Judah's and Shelah's~\cite{JSpres} and Brendle's~\cite{Breciclar} preservation theory.

\begin{definition}\label{def:good}
Let $\Rbf=\la X, Y, \sqsubset\ra$ be a Prs and let $\theta$ be an infinite cardinal.
A poset $\Por$ is \textit{$\theta$-$\Rbf$-good} if, for any $\Por$-name $\dot{h}$ for a member of $Y$, there is a non-empty set $H\subseteq Y$ (in the ground model) of size ${<}\theta$ such that, for any $x\in X$, if $x$ is $\Rbf$-unbounded over  $H$ then $\Vdash x\nsqsubset \dot{h}$.

We say that $\Por$ is \textit{$\Rbf$-good} if it is $\aleph_1$-$\Rbf$-good.
\end{definition}

Good posets allow us to preserve the Tukey order as follows.

\begin{lemma}\label{goodpres}
Let $\theta$ be regular uncountable, and let $\Rbf$ be a Prs. Assume that $\Por$ is a $\theta$-cc $\theta$-$\Rbf$-good poset. If $\mu$ is a cardinal, $\cf(\mu)\geq\theta$, $|I|\geq\mu$ and $\Cbf_{[I]^{<\mu}}\leqT\Rbf$, then $\Por$ forces that $\Cbf_{[I]^{<\mu}}\leqT\Rbf$.
\end{lemma}
\begin{proof}
Choose a sequence $\la x_i:\, i\in I \ra$ as in \autoref{lem:TukeyCtheta}~(b). We show that $\Por$ forces $|\{i\in I:\, x_i\sqsubset y\}|<\mu$ for all $y\in Y$. Let $\dot y$ be a $\Por$-name of a member of $Y$ and choose $H$ as in \autoref{def:good}. Let $B:=\bigcup_{y'\in H}\{i\in I:\, x_i\sqsubset y'\}$, so $|B|<\mu$. Since $\Por$ forces $x_i\sqsubset\dot y\Rightarrow i\in B$, then $\Por$ forces $|\{i\in I:\, x_i\sqsubset \dot y\}|<\mu$.
\end{proof}

We now present some examples of good posets. A general one is:

\begin{lemma}\label{smallgood}
If $\theta\geq\aleph_1$ regular and  $|\Por|<\theta$, then $\Por$ is $\theta$-$\Rbf$-good. In particular, Cohen forcing is $\Rbf$-good.
\end{lemma}
\begin{proof}
See e.g.~\cite[Lemma~4]{M}.
\end{proof}

\begin{example}\label{exm:good}
We indicate the type of posets that are good for the Prs of Cicho\'n's diagram, namely, those of \autoref{ExmPrs}.
\begin{enumerate}[label=(\arabic*)]
    \item Miller~\cite{Mi} showed that $\Eor$ is $\omega^\omega$-good. Also, random forcing is $\omega^\omega$-good.
    More generally, any $\mu$-$\mathrm{Fr}$-linked poset is $\mu^+$-$\Dbf$-good (see~\cite{Mvert,BCM} for details).
    \item Any $\mu$-centered poset is $\mu^+$-$\Cn$-good (see e.g.~\cite{Breciclar}). In particular, $\Eor$ and $\Dor$ are $\Cn$-good.
    
    \item Any $\mu$-centered poset is $\mu^+$-$\Lc^*$-good (see~\cite{Breciclar,JSpres}), so, in particular, $\Eor$ and $\Dor$ are $\Lc^*$-good.
    
    Besides,  Kamburelis~\cite{Ka} showed that any Boolean algebra with a sfam (strict finitely additive measure) is $\Lc^*$-good. In particular, any subalgebra of random forcing is $\Lc^*$-good.
\end{enumerate}
\end{example}

Good posets are preserved along FS iterations as follows.

\begin{theorem}\label{sizeforbd}
Any FS iteration of $\theta$-cc $\theta$-$\Rbf$-good posets is again $\theta$-$\Rbf$-good, when $\theta$ is regular uncountable.
\end{theorem}
\begin{proof}
See e.g.~\cite[Thm.~4.15]{CM19}.
\end{proof}

As a consequence, we get the following main result.

\begin{theorem}[Fuchino and the second author]\label{PreEUB}
Let $\theta$ be an uncountable regular cardinal.
If $\Por=\la\Por_\xi,\Qnm_\xi:\,  \xi<\nu\ra$ is a FS iteration of $\theta$-cc $\theta$-$\Rbf$-good posets, and $\nu\geq\theta$, 
then $\Por$ forces $\Cbf_{[\nu]^{<\theta}}\leqT\Rbf$. 

In particular, $\Por$ forces $\vartheta\leqT\Rbf$ for any regular $\theta\leq\vartheta\leq|\nu|$.
\end{theorem}
\begin{proof}
We only prove the particular case when $\nu=\lambda+\delta$ where $\lambda:=|\nu|$ and $\Qnm_\xi=\Cor$ for all $\xi<\lambda$. By \autoref{Cohenit}, $\Por_\lambda$ forces $\Cbf_{[\nu]^{<\aleph_1}}\leqT\Mg$, which implies $\Cbf_{[\nu]^{<\theta}}\leqT\Rbf$ by \autoref{ex:tukeysmall} and \autoref{Prsfact}. Since the remaining of the iteration is $\theta$-cc and $\theta$-$\Rbf$-good, by \autoref{goodpres} $\Por$ forces $\Cbf_{[\nu]^{<\theta}}\leqT\Rbf$. 

The ``in particular" follows by \autoref{impEUB}.
\end{proof}

\begin{remark}
In connection with the previous result, \cite{GKScicmax} defines the following property when $\vartheta$ is a limit ordinal:\footnote{The original notation is $\LCU$. The notation $\EUB$ (eventually unbounded) comes from~\cite{Bremodern}.}
\begin{description}
\item[$\EUB(\Rbf,\Por,\vartheta)$] states that there is a sequence $\la\dot x_\alpha:\alpha<\vartheta\ra$ of $\Por$-names of members of $X$ such that, for every $\Por$-name $\dot y$ of a member of $Y$, there exists an $\alpha_{\dot y}<\vartheta$ such that $\forall\, \beta \geq \alpha_{\dot y}\colon \Vdash\neg(\dot x_\beta \sqsubset \dot y)$.
\end{description}
This property is equivalent to $\COB(\Rbf^\perp,\Por,\cf(\vartheta),\cf(\vartheta))$, so it implies $\Vdash\vartheta\leqT\Rbf$ (see \autoref{remCOB}). In fact, equivalence holds when $\Por$ is $\cf(\vartheta)$-cc.
\end{remark}

\section{Applications to the left side}\label{aply}

This section is dedicated to forcing many values in Cicho\'n's diagram, particularly for the left side, by applying the methods of the previous sections.

From now on, we denote the Prs introduced in \autoref{ExmPrs} by $\Rbf_1:=\Lc^*\eqT\Nwf$, $\Rbf_2:=\Cn\eqT\Cbf_\Nwf^\perp$, $\Rbf_3:=\baire$, and $\Rbf_4:=\Mg\eqT\Cbf_\Mwf$.

\subsection{Warming up}

In this section, we present the effect on Cicho\'n's diagram after the FS iteration of the posets of \autoref{exm:Suslinccc}. We fix a cardinal $\lambda=\lambda^{\aleph_0}$.

\subsubsection*{Cohen forcing} 

After iterating Cohen forcing $\lambda$-many times, we obtain $\Cor_\lambda$. This forces $\lambda=\cfrak$ and, by \autoref{Cohenit}, $\Cbf_{[\lambda]^{<\aleph_1}}\leqT\Rbf_4$. On the other hand $\Rbf_1\leqT\Cbf_{[\R]^{<\aleph_1}}\eqT\Cbf_{[\lambda]^{<\aleph_1}}$ (see \autoref{fig:cichontukey}). Therefore, $\Cor_\lambda$ forces
$\Rbf_i\eqT \Cbf_{[\lambda]^{<\aleph_1}}$ for all $1\leq i\leq 4$. In particular, it forces the constellation of \autoref{EffectC}.

\begin{figure}[H]
\begin{center}
  \includegraphics[scale=1]{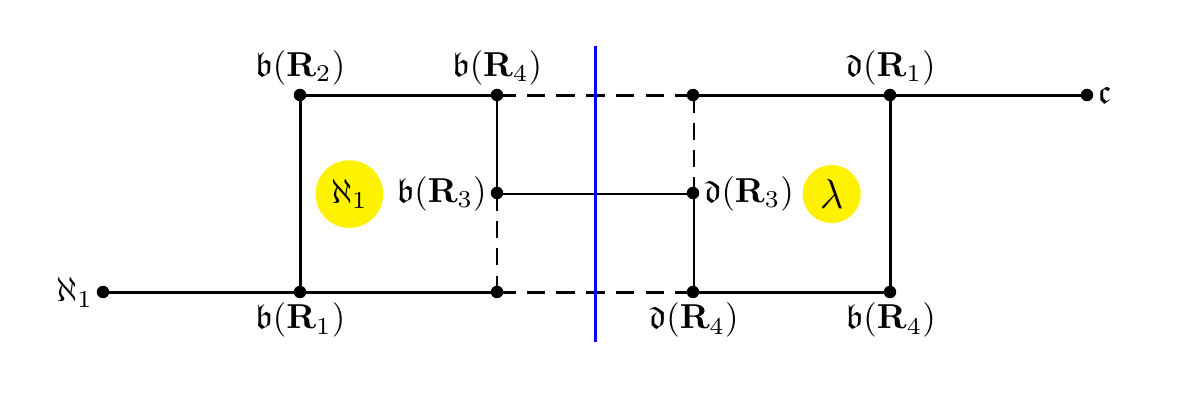}
  \caption{Cicho\'n's diagram constellation in Cohen's model.}
  \label{EffectC}
\end{center}
\end{figure}

\subsubsection*{Random forcing} 

Let $\Por$ be the FS iteration of $\Bor$ of length $\lambda$. Then $\Por$ forces
\begin{enumerate}[label=(\roman*)]
    \item $\Rbf_2\eqT\Rbf_4\eqT\lambda$, hence $\cov(\Nwf)=\non(\Mwf)=\cov(\Mwf)=\non(\Nwf)=\cf(\lambda)$; and
    \item $\Rbf_1\eqT\Rbf_3\eqT\Cbf_{[\lambda]^{<\aleph_1}}\eqT\Cbf_{[\R]^{<\aleph_1}}$, hence $\add(\Nwf)=\bfrak=\aleph_1$ and $\dfrak=\cof(\Nwf)=\cfrak=\lambda$.
\end{enumerate}
In particular, when $\lambda$ is regular, $\Por$ forces \autoref{EffectB}.
\begin{figure}[ht]
\begin{center}
  \includegraphics[scale=1]{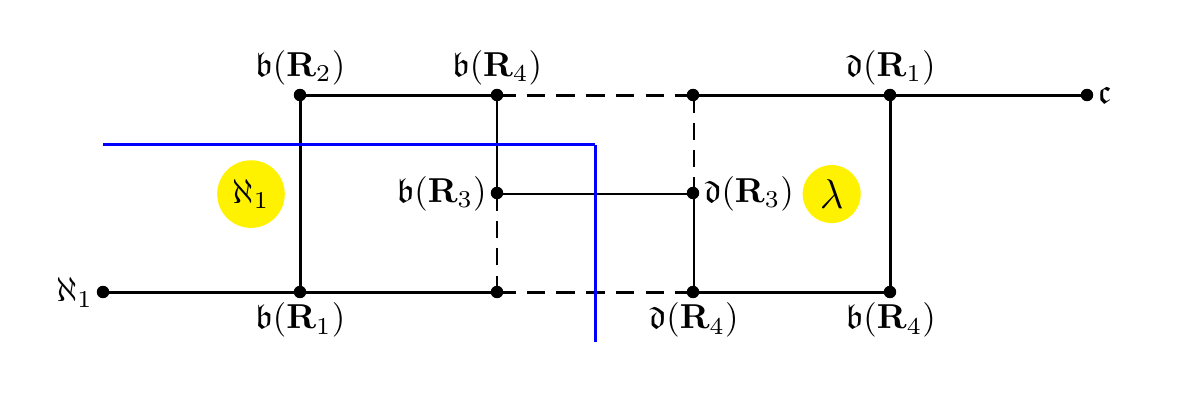}
  \caption{Constellation of Cicho\'n's diagram after a FS iteration of random forcing of length $\lambda$ regular.}
  \label{EffectB}
\end{center}
\end{figure}

Indeed, (i) follows by \autoref{thm:fullgen}, while (ii) follows by \autoref{exm:good}~(1) and by \autoref{PreEUB}, also because $\Por$ forces $\cfrak=\lambda$.

\subsubsection*{Eventually different reals forcing} 

Let $\Por$ be the FS iteration of $\Eor$ of length $\lambda$.  Then $\Por$ forces
\begin{enumerate}[label=(\roman*)]
    \item $\Rbf_4\eqT\lambda$, so $\non(\Mwf)=\cov(\Mwf)=\cf(\lambda)$; and
    \item $\Rbf_i\eqT\Cbf_{[\lambda]^{<\aleph_1}}\eqT\Cbf_{[\R]^{<\aleph_1}}$ for $1\leq i\leq3$, hence $\cov(\Nwf)=\bfrak=\aleph_1$ and $\dfrak=\non(\Nwf)=\cfrak=\lambda$.
\end{enumerate}
In particular, when $\lambda$ is regular, $\Por$ forces \autoref{EffectE}.

\begin{figure}[H]
\begin{center}
  \includegraphics[scale=1]{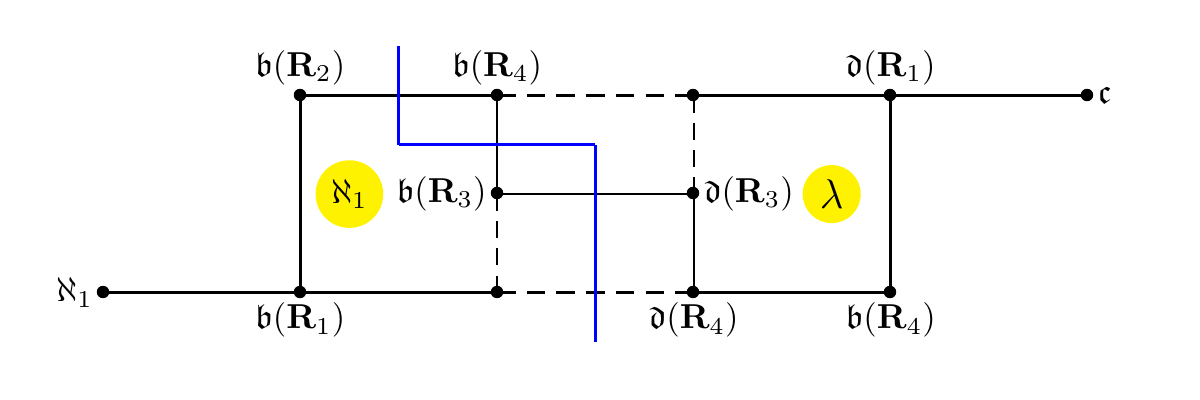}
  \caption{Constellation of Cicho\'n's diagram after a FS iteration of $\Eor$ of length $\lambda$ regular.}
  \label{EffectE}
\end{center}
\end{figure}

\subsubsection*{Hechler forcing}  

Let $\Por$ be the FS iteration of $\Dor$ of length $\lambda$. Then $\Por$ forces
\begin{enumerate}[label=(\roman*)]
    \item $\Rbf_3\eqT\Rbf_4\eqT\lambda$, so $\add(\Mwf)=\cof(\Mwf)=\cf(\lambda)$; and
    \item $\Rbf_1\eqT\Rbf_2\eqT\Cbf_{[\lambda]^{<\aleph_1}}\eqT\Cbf_{[\R]^{<\aleph_1}}$.
\end{enumerate}
In particular, when $\lambda$ is regular, $\Por$ forces \autoref{EffectD}.

\begin{figure}[ht]
\begin{center}
  \includegraphics[scale=1]{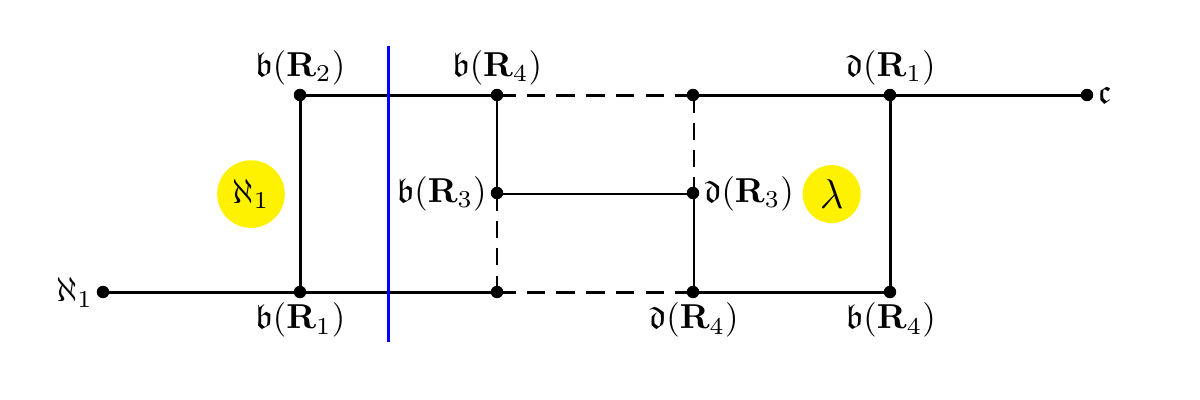}
  \caption{Constellation of Cicho\'n's diagram in Hechler's model when $\lambda$ is regular.}
  \label{EffectD}
\end{center}
\end{figure}

\subsubsection*{Localization forcing} 

Let $\Por$ be a FS iteration of $\LOCor$ of length $\lambda$. Then $\Por$ forces $\cfrak=\lambda$ and $\Rbf_1\eqT\lambda$, hence $\add(\Nwf)=\cof(\Nwf)=\cf(\lambda)$. In particular, when $\lambda$ is regular, $\Por$ forces \autoref{EffectLOC}.

\begin{figure}[ht]
\begin{center}
  \includegraphics[scale=1]{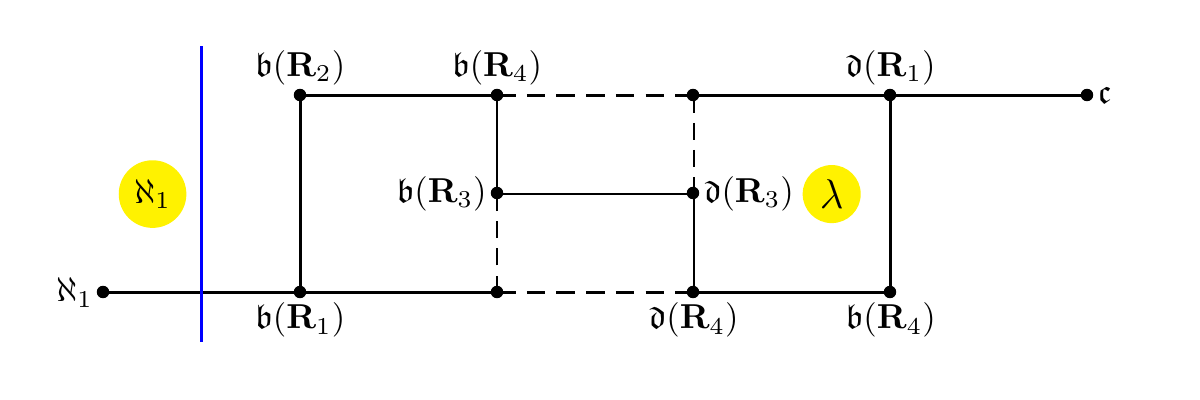}
  \caption{Constellation of Cicho\'n's diagram after a FS iteration of $\LOCor$ of length $\lambda$ regular.}
  \label{EffectLOC}
\end{center}
\end{figure}

\subsection{More values}

We now present examples of more different values in Cicho\'n's diagram. All the results cited from~\cite{M} are due to Brendle.

\begin{theorem}[{\cite[Theorem 3]{M}}]\label{mod1}
If $\aleph_1\leq\lambda_1\leq\lambda_2\leq\lambda_3\leq\lambda_4$ are regular cardinals and $\lambda_5\geq\lambda_4$ is a cardinal such that $\lambda_5^{<\lambda_3}=\lambda_5$, then there is a ccc poset forcing, for $1\leq i\leq 3$,
\begin{enumerate}[label=(\alph*)]
    \item $\Rbf_i\eqT\Cbf_{[\lambda_5]^{<{\lambda_i}}}\eqT[\lambda_5]^{<{\lambda_i}}$; and
    \item $\Rbf_4\eqT\lambda_4$. 
\end{enumerate}
In particular, we obtain the consistency of \autoref{BrendleJudahShelah}.
\begin{figure}[ht]
\begin{center}
  \includegraphics[scale=1]{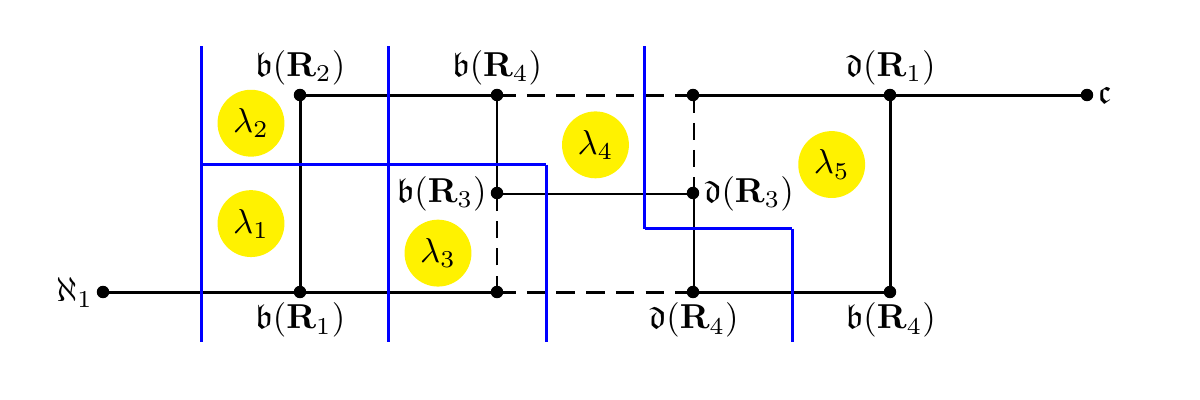}
  \caption{Six values in Cicho\'n's diagram.}
  \label{BrendleJudahShelah}
\end{center}
\end{figure}
\end{theorem}
\begin{proof}
We shall perform a FS iteration $\Por=\la\Por_\xi,\Qnm_\xi:\,  \xi<\nu\ra$ of length $\nu:=\lambda_5\lambda_4$ (ordinal product) as follows. Fix a partition $\la C_i:\,  1\leq i\leq 3\ra$ of $\lambda_5\smallsetminus\{0\}$ where each set has size $\lambda_5$. For each $\rho<\lambda_4$ denote $\eta_\rho:=\lambda_5\rho$. We define the iteration at each $\xi=\eta_\rho+\varepsilon$ for $\rho<\lambda_4$ and $\varepsilon<\lambda_5$ as follows (see \autoref{FSrestrictsmallposet}):

\[\Qnm_\xi:=\left\{\begin{array}{ll}
        \dot\Eor   & \text{if $\varepsilon=0$,}\\
        \LOCor^{\dot N_\xi}  & \text{if $\varepsilon\in C_1$,} \\
        \Bor^{\dot N_\xi} & \text{if $\varepsilon\in C_2$,} \\ 
        \Dor^{\dot N_\xi} & \text{if $\varepsilon\in C_3$,}
    \end{array}\right.\]
where $\dot N_\xi$ is a $\Por_{\xi}$-name of a transitive model of $\thzfc$ of size ${<}\lambda_i$ when $\varepsilon\in C_i$. 

Additionally, by a book-keeping argument, we make sure that all such models $N_\xi$ are constructed such that, for any $\rho<\lambda_4$:
\begin{enumerate}[label=(\roman*)]
    \item if $A\in V_{\eta_\rho}$ is a subset of $\omega^\omega$ of size ${<}\lambda_1$, then there is some $\varepsilon\in C_1$ such that  $A\subseteq N_{\eta_\rho+\varepsilon}$;
    
     \item if $A\in V_{\eta_\rho}$ is a subset of $\Omega$ of size ${<}\lambda_2$, then there is some $\varepsilon\in C_2$ such that  $A\subseteq N_{\eta_\rho+\varepsilon}$; and

    \item if $A\in V_{\eta_\rho}$ is a subset of $\omega^\omega$ of size ${<}\lambda_3$, then there is some $\varepsilon\in C_3$ such that $A\subseteq N_{\eta_\rho+\varepsilon}$.   
\end{enumerate}

\begin{figure}[ht]
\begin{center}
\includegraphics[scale=1]{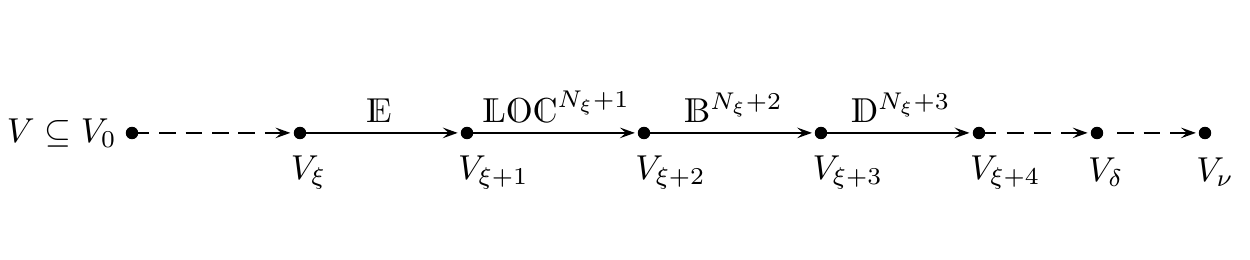}
\caption{A FS iteration of length $\nu$ of ccc partial orders, going through $\Eor$ cofinally often, as well as through all subforcings of localization forcing of size ${<}\lambda_1$, all subforcings of random forcing of size ${<}\lambda_2$, and all subforcings of Hechler forcing of size ${<}\lambda_3$.}
\label{FSrestrictsmallposet}
\end{center}
\end{figure}

We prove that $\Por$ is as required. Clearly, $\Por$ forces $\cfrak=|\nu|=\lambda_5$.\smallskip

Fix $1\leq i\leq 3$. Note that all iterands are $\lambda_i$-$\Rbf_i$-good (see \autoref{smallgood} and \autoref{exm:good}), hence, by \autoref{PreEUB}, $\Por$ forces $\Cbf_{[\nu]^{<{\lambda_i}}}\leqT\Rbf_i$. On the other hand, $\Por$ forces $\Rbf\leqT\Cbf_{[\R]^{<{\lambda_i}}}$ by \autoref{itsmallsets}. Therefore, since $|\Rbf|=|\nu|=\lambda_5$, we conclude (a).\smallskip

Finally, since $\cf(\nu)=\lambda_4$, by \autoref{thm:fullgen} $\Por$ forces $\Rbf_4\eqT\lambda_4$.
\end{proof}

\begin{theorem}[{\cite[Theorem 2]{M}}]
If $\aleph_1\leq\lambda_1\leq\lambda_2\leq\lambda_3$ are regular cardinals and $\lambda_4\geq\lambda_3$ is a cardinal such that $\lambda_4^{<\lambda_3}=\lambda_4$, then there is a ccc poset forcing, for $1\leq i\leq 3$,
\begin{enumerate}[label=(\roman*)]
    \item $\Rbf_i\eqT\Cbf_{[\lambda_4]^{<{\lambda_i}}}\eqT[\lambda_4]^{<{\lambda_i}}$; and 
    \item $\Rbf_4\eqT\Rbf_3\eqT\Cbf_{[\lambda_4]^{<{\lambda_3}}}\eqT[\lambda_4]^{<{\lambda_3}}$.
\end{enumerate}
In particular, we obtain the consistency of \autoref{MejiaMatrixthm2}.
\begin{figure}[ht]
\begin{center}
  \includegraphics[scale=1]{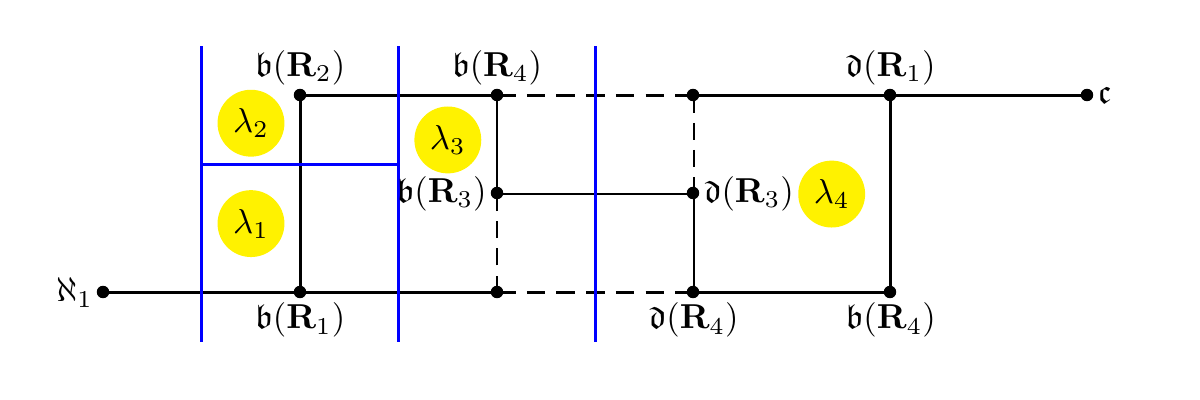}
  \caption{Five values in Cicho\'n's diagram.}
  \label{MejiaMatrixthm2}
\end{center}
\end{figure}
\end{theorem}
\begin{proof}
Perform a FS iteration $\Por=\la\Por_\xi,\Qnm_\xi:\,  \xi<\lambda_4\ra$ as follows. Fix a partition $\la C_i:\,  1\leq i\leq 3\ra$ of $\lambda_4$ into cofinal subsets of size $\lambda_4$.
For each $\xi<\lambda_5$ define:
\[\Qnm_\xi:=\left\{\begin{array}{ll}
        \LOCor^{\dot N_\xi}  & \text{if $\xi\in C_1$,} \\
        \Bor^{\dot N_\xi} & \text{if $\xi\in C_2$,} \\ 
        \Dor^{\dot N_\xi} & \text{if $\xi\in C_3$,}
    \end{array}\right.\]
where $\dot N_\xi$ is a $\Por_\xi$-name of a transitive model of $\thzfc$ of size ${<}\lambda_i$ when $\xi\in C_i$.
Additionally, by a book-keeping argument, we make sure that all such models $N_\xi$ are constructed such that conditions similar to (i)--(iii) of the proof of \autoref{mod1} are satisfied. Concretely, if we denote $\Rbf_i:=\la X_i,Y_i,\sqsubset^i\ra$, we guarantee that, for any $\xi<\lambda$ and $A\subseteq X_i^{V_\xi}$ of size ${<}\lambda_i$, there is some $\eta\geq\xi$ in $C_i$ such that $A\subseteq N_\eta$. Then, $\Por$ is as required.
\end{proof}

\begin{theorem}[{\cite[Theorem 4]{M}}]\label{mod3}
If $\aleph_1\leq\lambda_1\leq\lambda_2\leq\lambda_3$ are regular cardinals and $\lambda_4\geq\lambda_3$ is a cardinal such that $\lambda_4^{<\lambda_2}=\lambda_4$, then there is a ccc poset forcing
\begin{enumerate}[label=(\roman*)]
    \item $\Rbf_1\eqT\Cbf_{[\lambda_4]^{<{\lambda_1}}}\eqT[\lambda_4]^{<{\lambda_1}}$, $\Rbf_3\eqT\Cbf_{[\lambda_4]^{<{\lambda_2}}}\eqT[\lambda_4]^{<{\lambda_2}}$, and
    \item $\Rbf_2\eqT\Rbf_4\eqT\lambda_3$. 
\end{enumerate}
In particular, we obtain the consistency of \autoref{MejiaMatrixthm4}.
\begin{figure}[H]
\begin{center}
  \includegraphics[scale=1]{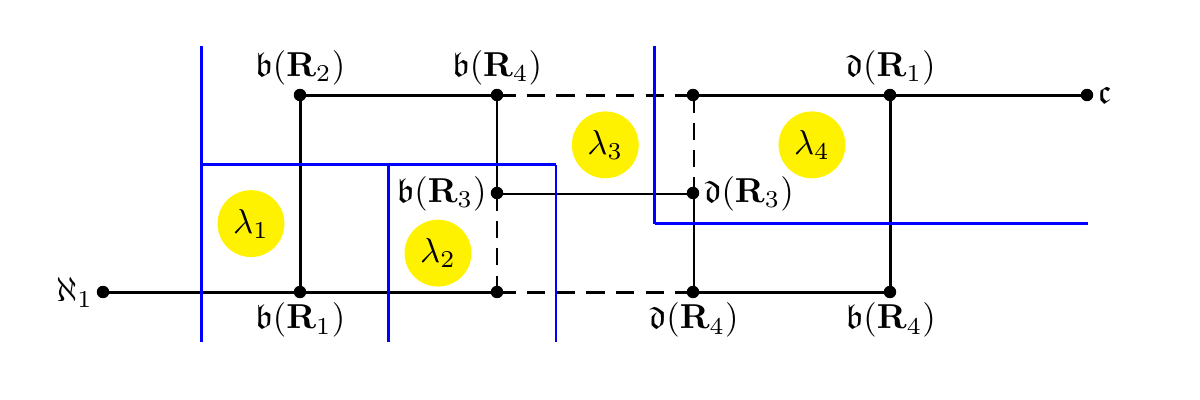}
  \caption{Five values in Cicho\'n's diagram.}
  \label{MejiaMatrixthm4}
\end{center}
\end{figure}
\end{theorem}
\begin{proof}
Perform a FS iteration $\la\Por_\xi,\Qnm_\xi:\,  \xi<\nu\ra$ of length $\nu:=\lambda_4\lambda_3$ as follows. Fix a partition $\la C_i:\,  1\leq i\leq 3\ra$ of $\lambda_4\smallsetminus\{0\}$ into cofinal subsets of size $\lambda_4$.
For each $\rho<\lambda_3$ denote $\eta_\rho:=\lambda_4\rho$. We define the iteration at each $\xi=\eta_\rho+\varepsilon$ for $\rho<\lambda_3$ and $\varepsilon<\lambda_4$ as follows:
\[\Qnm_\xi:=\left\{\begin{array}{ll}
        \LOCor^{\dot N_\xi}  & \text{if $\varepsilon\in C_1$,} \\
        \Dor^{\dot N_\xi} & \text{if $\varepsilon\in C_2$,}\\
        \dot{\Bor} & \text{if $\varepsilon\in C_3$,}
    \end{array}\right.\]
where $\dot N_\xi$ is a $\Por_{\xi}$-name of a transitive model of $\thzfc$ of size ${<}\lambda_i$ when $\varepsilon\in C_i$. We use book-keeping as in (i)--(iii) of \autoref{mod1}.
The poset $\Por$ is as required.
\end{proof}

\begin{theorem}[{\cite[Theorem 5]{M}}]
If $\aleph_1\leq\lambda_1\leq\lambda_2\leq\lambda_3$ are regular cardinals and $\lambda_4\geq\lambda_3$ is a cardinal such that $\lambda_4^{<\lambda_2}=\lambda_4$, then there is a ccc poset forcing, for $1\leq i\leq 2$
\begin{enumerate}[label=(\roman*)]
    \item $\Rbf_i\eqT\Cbf_{[\lambda_4]^{<{\lambda_i}}}\eqT[\lambda_4]^{<{\lambda_i}}$; and
    \item $\Rbf_3\eqT\Rbf_4\eqT\lambda_3$.
\end{enumerate}
In particular, we obtain the consistency of \autoref{MejiaMatrixthm5}.
\begin{figure}[ht]
\begin{center}
  \includegraphics[scale=1]{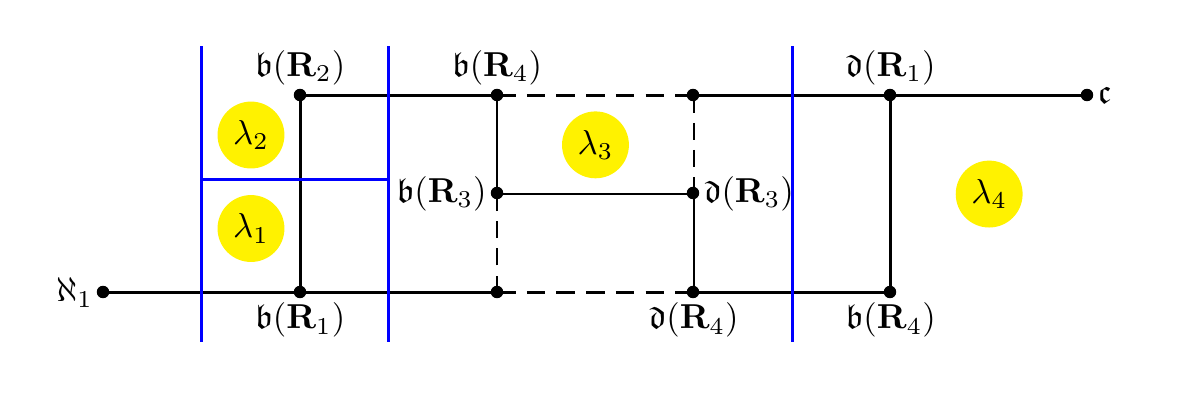}
  \caption{Five values in Cicho\'n's diagram.}
  \label{MejiaMatrixthm5}
\end{center}
\end{figure}
\end{theorem}
\begin{proof}
Perform a FS iteration $\Por=\la\Por_\xi,\Qnm_\xi:\,  \xi<\nu\ra$ of length $\nu:=\lambda_4\lambda_3$ as follows. Consider the same preparation as in the proof of \autoref{mod3}. Using book-keeping as in previous proofs, we define the iteration at each $\xi=\eta_\rho+\varepsilon$ for $\rho<\lambda_3$ and $\varepsilon<\lambda_4$ as follows:
\[\Qnm_\xi:=\left\{\begin{array}{ll}
        \LOCor^{\dot N_\xi}  & \text{if $\varepsilon\in C_1$,}\\
        \Bor^{\dot N_\xi} & \text{if $\varepsilon\in C_2$,}\\
        \dot{\Dor} & \text{if $\varepsilon\in C_3$,}
    \end{array}\right.\]
where $\dot N_\xi$ is a $\Por_{\xi}$-name of a transitive model of $\thzfc$ of size ${<}\lambda_{i}$ when $\varepsilon\in C_i$.  
%
%
\end{proof}






We conclude this section by presenting three important results of the left-hand side of Cicho\'n's digram, which uses sophisticated techniques such as finitely additive measures as well as ultrafilters along FS iterations, and ultrafilters along matrix iterations. 

\begin{theorem}[{\cite{GMS,GKScicmax}}]\label{gksmax}
Let $\lambda_1\leq\lambda_2\leq\lambda_3=\lambda_3^{<\lambda_3}\leq\lambda_4$ be uncountable regular cardinals, and assume that $\lambda_4<\lambda_5^{<\lambda_4}$ and  $\lambda_4$ is $\aleph_1$-inaccessible.\footnote{A cardinal $\lambda$ is \emph{$\theta$-inaccessible} if $\mu^\nu<\lambda$ for any $\mu<\lambda$ and $\nu<\theta$.} Then there is a ccc poset that forces $\cfrak=\lambda_5$ and $\Rbf_i\eqT\Cbf_{[\lambda_5]^{<\lambda_i}}\eqT[\lambda_5]^{<\lambda_i}\eqT[\lambda_5]^{<\lambda_i}\cap V$ for all $1\leq i\leq 4$. In particular, it forces the constellation in \autoref{GMS}.
\begin{figure}[H]
\begin{center}
  \includegraphics[scale=1]{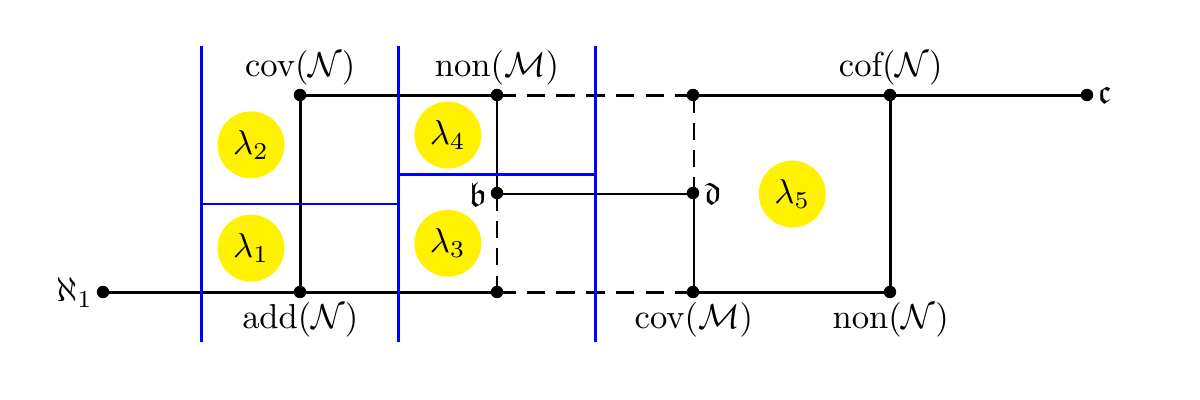}
  \caption{The left side of Cicho\'n's diagram.}
  \label{GMS}
\end{center}
\end{figure}
\end{theorem}

\begin{theorem}[{\cite{KST}}]\label{leftKST}
Let $\lambda_1\leq\lambda_2=\lambda_2^{<\lambda_2}\leq\lambda_3\leq\lambda_4$ be regular cardinals, and assume that $\lambda_3$ and $\lambda_4$ are $\aleph_1$-inaccessible, and $\lambda_5=\lambda_5^{<\lambda_4}>\lambda_4$. Then there is a ccc poset that forces $\cfrak=\lambda_5$, $\Rbf_i\eqT[\lambda_5]^{<\lambda_i}\cap V$ for $i=1,4$, $\Rbf_2\eqT[\lambda_5]^{<\lambda_3}\cap V$ and $\Rbf_3\eqT[\lambda_5]^{<\lambda_2}\cap V$. In particular, it forces the constellation in \autoref{KST}.
\end{theorem}

\begin{figure}[ht]
\begin{center}
  \includegraphics[scale=1]{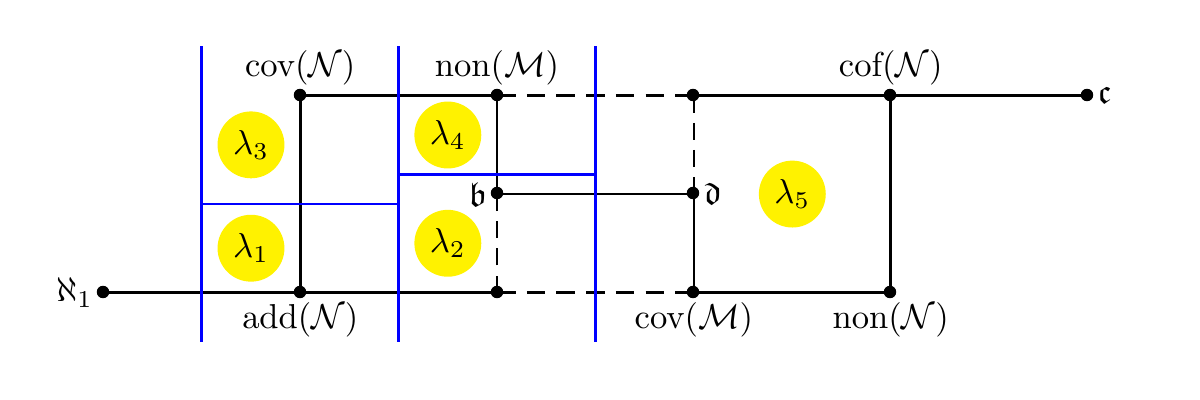}
  \caption{Alternative left side of Cicho\'n's diagram.}
  \label{KST}
\end{center}
\end{figure}

\begin{theorem}[{\cite{BCM}}]\label{leftBCM}
 Let $\lambda_0\leq\lambda_1\leq\lambda_2\leq\lambda_3\leq\lambda_4\leq\lambda_5$ be uncountable regular cardinals and let $\lambda_6\geq\lambda_5$ be a cardinal such that $\lambda_6^{{<}\lambda_3}=\lambda_6$, then there is a ccc poset that forces 
 \begin{enumerate}[label=(\alph*)]
     \item $\Rbf_i\eqT\Cbf_{[\lambda_6]^{<\lambda_i}\cap V}\eqT[\lambda_6]^{<\lambda_i}\cap V$ for $1\leq i\leq 3$, and
     \item $\lambda_4\leqT\Rbf_4$, $\lambda_5\leqT\Rbf_4$ and $\Rbf_4\leqT\lambda_5\times\lambda_4$.
 \end{enumerate}
 In particular, it forces the constellation in \autoref{BCM}.
\end{theorem}
\begin{figure}[H]
\begin{center}
  \includegraphics[scale=1]{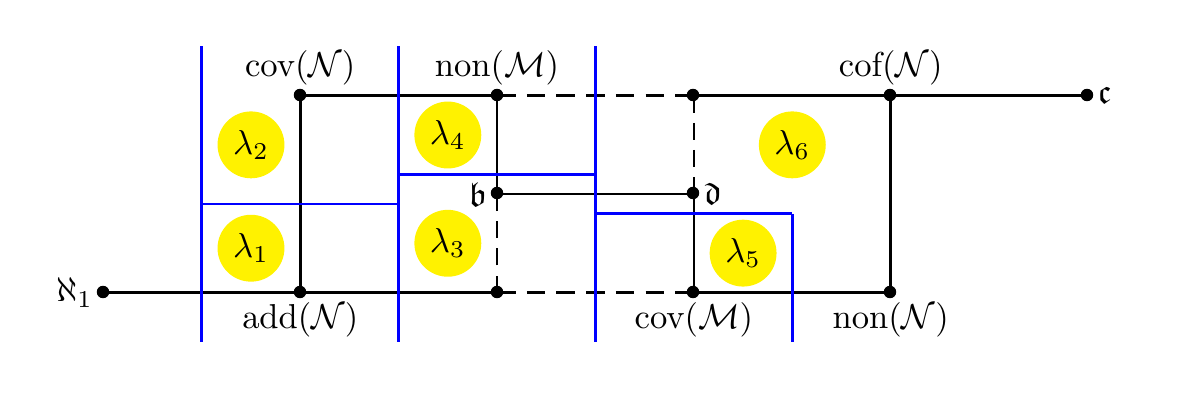}
  \caption{Seven values in Cicho\'n's digram.}
  \label{BCM}
\end{center}
\end{figure}

We remark that, in \autoref{leftBCM}, we cannot force $\Rbf_4\eqT\Cbf_{[\lambda_5]^{<\lambda_4}}$ because, in the ground model, $\Cbf_{[\lambda_5]^{<\lambda_4}}\not\leqT\lambda_5\times\lambda_4$ in the case when $\lambda_4<\lambda_5$. To see this, note that if $\la (a_i,b_i):\, i<\lambda_5\ra\subseteq\lambda_5\times\lambda_4$, then there is some $L\subseteq\lambda_5$ of size $\lambda_5$ such that the sequence $\la b_i:\, i\in L\ra$ is constant with value some $b<\lambda_4$. Then, it is possible to find some $a<\lambda_5$ such that $\{i\in L:\, (a_i,b_i)\leq(a,b)\}$ has size ${\geq}\lambda_4$, so we conclude that there is no Tukey connection by \autoref{lem:TukeyCtheta}~(b). On the other hand, we can say that $\Rbf_4\leqT\Cbf_{[\lambda_5]^{<\lambda_4}\cap V}$ because, in the ground model, $\lambda_5\times\lambda_4\leqT\Cbf_{[\lambda_5\times\lambda_4]^{<\lambda_4}}$.

\section{Restriction to submodels}\label{sec:restrmod}

We present the general theory of intersection of posets with $\sigma$-closed models. This is the main tool in~\cite{GKMS} to force Cicho\'n's maximum without using large cardinals. In this section we do not only review this method, but we analyze its effect on the Tukey order.
For this section, we fix:
\begin{enumerate}[label=(F\arabic*)]
    \item a ccc poset $\Por$;
    \item a definable relational system $\Rbf=\la X,Y,\sqsubset\ra$ of the reals, here wlog $X=Y=\omega^\omega$; and
    \item a large enough regular cardinal $\chi$ such that $\Por\in H_\chi$, and $H_\chi$ contains all the parameters defining $\Rbf$.
\end{enumerate}
 
\begin{definition}
  A model $N\preceq H_\chi$ is \emph{${<}\kappa$-closed} if $N^{<\kappa}\subseteq N$. 
  We write \emph{$\sigma$-closed} for ${<}\aleph_1$-closed.
\end{definition}

When intersecting a ccc poset with a $\sigma$-closed model, we obtain a completely embedded subforcing.

\begin{lemma}
 If $N\preceq H_\chi$ is $\sigma$-closed and $\Por\in N$, then $\Por\cap N\lessdot\Por$.
\end{lemma}

Semantically, there is a correspondence between some $\Por\cap N$-names a $\Por$-names belonging to $N$, and we can also have a correspondence for the forcing relation for some formulas.

\begin{fact}
  If $\kappa>\aleph_0$ is regular and $N$ is ${<}\kappa$ closed then there is a one-to-one correspondence between:
  \begin{enumerate}[label=(\roman*)]
    \item $\Por$-names $\tau\in N$ and
    \item $\Por\cap N$-names $\sigma$
  \end{enumerate}
  of members of $H_\kappa$ (in particular, reals).  
  Thus, if $G$ is $\Por$-generic over $V$ then $N[G]\cap H_\kappa^{V[G]} =  H_\kappa^{V[G\cap N]}$.
\end{fact}

\begin{corollary}
  For absolute $\varphi(\bar x)$ (e.g. Borel on the reals) if $p\in\Por\cap N$ and $\bar\tau\in N$ is a finite sequence of $\Por$-names of members of $H_\kappa$, then
  \[p\Vdash_\Por \varphi(\bar\tau) \sii p\Vdash_{\Por\cap N}\varphi(\bar\sigma).\]
\end{corollary}

The following result illustrates the main motivation to intersect ccc posets with $\sigma$-closed models, since it affects the Tukey relations forced by the posets.

\begin{lemma}\label{fct:restrN}
  Let $N\preceq H_\chi$ be $\sigma$-closed and let $K=\la A,B,\lhd\ra$ be a relational system. Assume that $\Por$, $K$ and the parameters of $\Rbf$ are in $N$.
  \begin{enumerate}[label=(\alph*)]
    \item If $\Por\Vdash\Rbf\leqT K$ then $\Por\cap N\Vdash \Rbf\leqT K\cap N$ where $K\cap N:=\la A\cap N,B\cap N,\lhd\ra$.
    \item If $\Por\Vdash K\leqT \Rbf$ then $\Por\cap N\Vdash K\cap N\leqT \Rbf$.
  \end{enumerate}
\end{lemma}
\begin{proof}
\noindent(a): Find a sequence $\la\dot y_j:\, j\in B\ra\in N$ of $\Por$-names of members of $\omega^\omega$ such that
\[\Vdash_\Por\forall\, x\in \omega^\omega\ \exists\, i_x\in A\ \forall\, j\in B\colon i_x\lhd j \imp x \sqsubset \dot y_j.\]
For $j\in B\cap N$, $\dot y_j$ can be seen as a $\Por\cap N$-name of a member of $\omega^\omega$.

We claim that
$\Vdash_{\Por\cap N}\forall\, x\in \omega^\omega\ \exists\, i_x\in A\cap N\ \forall\, j\in B\cap N\colon i_x\lhd j \imp x \sqsubset \dot y_j$.
Let $p\in\Por\cap N$ and let $\dot x$ be a $\Por\cap N$-name of a real. Then $\dot x\in N$ and
\[N\models p\Vdash_\Por \exists\, i_{\dot x}\in A\ \forall\, j\in B\colon i_{\dot x}\lhd j \imp \dot x \sqsubset \dot y_j.\] 
Find $q\leq p$ in $N$ and $i_{\dot x}\in A\cap N$ such that
\[\forall\, j\in B\cap N \colon i_{\dot x}\lhd j \imp N\models q\Vdash_\Por \dot x \sqsubset \dot y_j.\] 
But $N\models q\Vdash_\Por \dot x \sqsubset \dot y_j \sii q\Vdash_\Por \dot x \sqsubset \dot y_j\sii q\Vdash_{\Por\cap N} \dot x \sqsubset \dot y_j$.  
Thus $\la \dot y_j:\, j\in B\cap N\ra$ witnesses $\Por\cap N\Vdash\Rbf\leqT K\cap N$.\smallskip

\noindent(b): If $\Por\Vdash K\leqT \Rbf$ then $\Por\Vdash \Rbf^\perp\leqT K^\perp$. Although $\Rbf$ is not as in \autoref{def:defrel}, the relation $\sqsubset$ is absolute enough to prove $\Por\Vdash \Rbf^\perp\leqT K^\perp\cap N$ as in~(a). Note that $K^\perp\cap N=\la B\cap N, A\cap N, \ntriangleright\ra=(K\cap N)^\perp$, so we can conclude that $\Por\cap N\Vdash K\cap N\leqT \Rbf$.
\end{proof}

As a consequence, if $\Por\Vdash\Rbf\eqT K$ then $\Por\cap N\Vdash \Rbf\eqT K\cap N$. Hence, to know the values that $\Por\cap N$ force to $\bfrak(\Rbf)$ and $\dfrak(\Rbf)$, we need to calculate the cardinal characteristics of the relational system $K\cap N$ (recall \autoref{cccgenext0}), or find natural Tukey equivalent relational systems. The theory developed from now on has the purpose to understand $K\cap N$ in some specific contexts. In the applications, $K$ is often a directed preorder.

\begin{fact}\label{fct:ScapN}
  Let $N\preceq H_\chi$ be $\sigma$-closed and let $K=\la A,B,\lhd\ra\in N$ be a relational system.
  \begin{enumerate}[label=(\alph*)]
    \item $\bfrak(K\cap N)\leq|\bfrak(K)\cap N|$ and $\dfrak(K\cap N)\leq|\dfrak(K)\cap N|$.
    
    \item If $N$ is ${<}\kappa$-closed then $\bfrak(K\cap N)\geq\min\{\bfrak(K),\kappa\}$.    
          In particular, if $N$ is ${<}\bfrak(K)$-closed then $\bfrak(K\cap N)=\bfrak(K)$.
          
    \item Property (b) holds for the $\dfrak$-numbers.
          
     \item If $S\in N$ is a directed preorder and $\dfrak(S)\subseteq N$, then $S\cap N\eqT S$.
  \end{enumerate}
\end{fact}
\begin{proof}
\noindent(a) Find $f\in N$, $f\colon \dfrak(K)\to B$ where $f[\dfrak(K)]$ is $K$-dominating.  
  So $\{f(\alpha):\, \alpha\in\dfrak(K)\cap N\}\subseteq B\cap N$ is $K\cap N$-dominating. Thus $\dfrak(K\cap N)\leq|\dfrak(K)\cap N|$. The inequality of the $\bfrak$-number follows by applying the previous to $K^\perp$.\smallskip

\noindent(b): If $F\subseteq A\cap N$ has size ${<}\min\{\bfrak(S),\kappa\}$ then
  $F\in N$ and
  $N\models \exists\, y\in B\ \forall\, x\in F\colon x\lhd y$,
  so such a $y$ can be found in $B\cap N$.\smallskip
  
\noindent(c): Apply (b) to $K^\perp$.\smallskip

\noindent(d): If $\dfrak(S)\subseteq N$ then $f[\dfrak(S)]\subseteq N$ where $f$ is as in (a), so $S\cap N$ is cofinal in $S$ and $S\cap N\eqT S$.
\end{proof}

\autoref{fig:collN} illustrates the situation of \autoref{fct:ScapN} when $K=S$ is a directed poset, 
$\delta_N:=\min\{\delta\in\On:\, \delta\notin N\}$ and $|N|<\delta_N$ (the latter will hold in our applications).

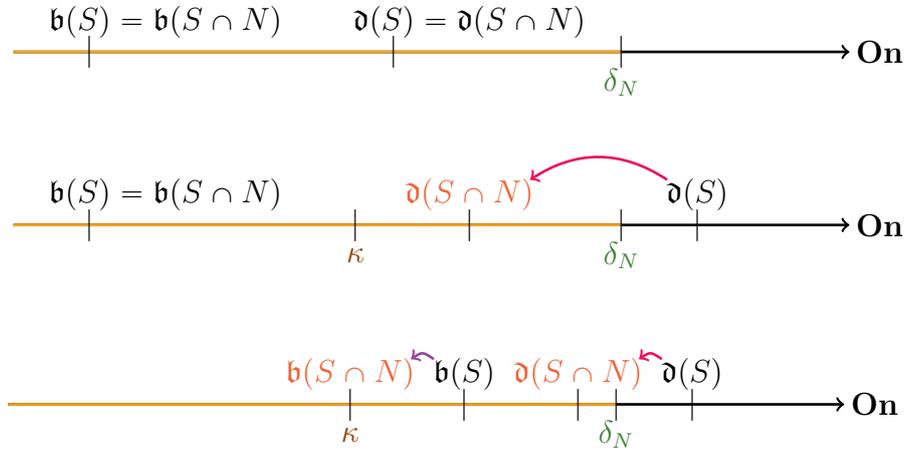
\begin{figure}[H]
\centering
\begin{tikzpicture}[line width=1pt]
\draw[->] (-1,0) -- (10,0);
\node at (10.4,0) {$\On$};
\draw[orange] (-1,0)--(7,0);
\node at (7,-0.4) {$\colive{ \delta_N}$};
\node at (7,0) {$|$};
\node at (5,0.4) {$\dfrak(S)=\dfrak(S\cap N)$};
\node at (4,0) {$|$};
\node at (1,0.4) {$\bfrak(S)=\bfrak(S\cap N)$};
\node at (0,0) {$|$};
\end{tikzpicture}

\vspace{15pt}

\begin{tikzpicture}[line width=1pt]
\draw[->] (-1,0) -- (10,0);
\node at (10.4,0) {$\On$};
\draw[orange] (-1,0)--(7,0);
\node at (7,-0.4) {$\colive{ \delta_N}$};
\node at (7,0) {$|$};
\node at (3.5,-0.4) {$\cmarron{\kappa}$};
\node at (3.5,0) {$|$};
\node at (1,0.4) {$\bfrak(S)=\bfrak(S\cap N)$};
\node at (0,0) {$|$};
\node at (8,0.4) {$\dfrak(S)$};
\node at (8,0) {$|$};
\node at (5,0.4) {$\cladri{\dfrak(S\cap N)}$};
\node at (5,0) {$|$};
\draw[fresa,->] (7.6,0.6) .. controls (7,1) and (6.4,1) .. (5.8,0.6);
\end{tikzpicture}

\vspace{25pt}

\begin{tikzpicture}[line width=1pt]
\draw[->] (-1,0) -- (10,0);
\node at (10.4,0) {$\On$};
\draw[orange] (-1,0)--(7,0);
\node at (7,-0.4) {$\colive{ \delta_N}$};
\node at (7,0) {$|$};
\node at (3.5,-0.4) {$\cmarron{\kappa}$};
\node at (3.5,0) {$|$};
\node at (3.5,0.4) {$\cladri{\bfrak(S\cap N)}$};
\node at (5,0.4) {$\bfrak(S)$};
\node at (5,0) {$|$};
\node at (8,0.4) {$\dfrak(S)$};
\node at (8,0) {$|$};
\node at (6.5,0.4) {$\cladri{\dfrak(S\cap N)}$};
\node at (6.5,0) {$|$};
\draw[fresa,->] (7.6,0.6) .. controls (7.45,0.7) .. (7.3,0.6);
\draw[purple,->] (4.6,0.6) .. controls (4.45,0.7) .. (4.3,0.6);
\end{tikzpicture}
\caption{Effect on $\bfrak(S)$ and $\dfrak(S)$ after intersecting with $N$ as in \autoref{fct:ScapN} when $|N|<\delta_N$ and $S=K$ is a directed preorder. The situation on the top corresponds to~(d), where the cardinal characteristics do not change; the middle corresponds to~(b), where $\bfrak(S\cap N)=\bfrak(S)$ but $\dfrak(S\cap N)$ gets smaller; and the situation at the bottom indicates that $\dfrak(S\cap N)$ gets smaller and that $\bfrak(S\cap N)$ may become smaller.}\label{fig:collN}
\end{figure}

The effect on the relational systems depends very much on the structure of the model. We look at models constructed from directed systems of models, as follows.

\begin{definition}
  Let $\kappa$ and $\theta$ be infinite cardinals, $\kappa$ uncountable regular, and let $T$ be a directed partial order without maximum.
  A sequence $\bar{N}:=\la N_t :\, t\in T\ra$ of elementary submodels of $H_\chi$ is a \emph{$(T,\kappa,\theta)$-directed system} if, for all $t\in T$:
  \begin{enumerate}[label=(\arabic*)]
    \item $N_t$ is ${<}\kappa$-closed and $|N_t|=\theta$;
    \item if $t\leq t'$ in $T$ then $N_t\subseteq N_{t'}$;
    \item $\theta\cup\{\theta,T\}\subseteq N_t$.
  \end{enumerate}
  In this context, we usually denote $N:=\bigcup_{t\in T}N_t$. Clearly $N\preceq H_\chi$.
\end{definition}

\begin{fact}
If $\bar{N}:=\la N_t :\, t\in T\ra$ is a $(T,\kappa,\theta)$-directed system then:
\begin{enumerate}[label=(\alph*)]
  \item $\theta^{<\kappa}=\theta$ (so $\kappa\leq\theta$).
  \item $N$ is ${<}\min\{\kappa,\bfrak(T)\}$-closed.
  \item $\theta\leq|N|\leq\theta\cdot\dfrak(T)$.
\end{enumerate}
\end{fact}
\begin{proof}
  To see (c) note that, if $T'\subseteq T$ witnesses $\dfrak(T)$, then $N=\bigcup_{t\in T'}N_t$.
\end{proof}

In our applications $|T|\leq\theta$, in which case (c) implies $|N|=\theta$.

The remaining results in this section are the main tools to understand $K\cap N$ when $N$ is obtained from a directed system.

\begin{lemma}\label{mainlem1}
  Let $\bar N=\la N_t:\, t\in T\ra$ be a $(T,\kappa,\theta)$-directed system.
  If $K=\la A,B,\lhd\ra$ is a relational system, $K\in N$, and
  \begin{itemize}
    \item[$(\heartsuit)$] $A\cap N_t$ is $K\cap N$-bounded for all $t\in T$,
  \end{itemize}
  then $K\cap N\leqT T$. 
\end{lemma}
\begin{proof}
  Define $f\colon A\cap N\to T$ such that $i\in A\cap N_{f(i)}$ for $i\in A\cap N$, and
  define $g\colon T\to B\cap N$ such that $\forall\, i\in A\cap N_t\colon i\lhd g(t)$ (by $(\heartsuit)$).  
  If $f(i)\leq_T t$ then $i\in A\cap N_{f(i)}\subseteq A\cap N_t$, so $i\lhd g(t)$.
\end{proof}

\begin{fact}\label{largeb}
  In \autoref{mainlem1}, we must have $\bfrak(K)>\theta$.
\end{fact}
\begin{proof}
  Since $K\in N$, $\exists\, t\in T\colon K\in N_t$,
  so we can find a witness $F\in N_t$ of $\bfrak(K)$.  
  If $|F|=\bfrak(S)\leq\theta$ then $F\subseteq A\cap N_t$ (because $\theta\subseteq N_t$), 
  so $A\cap N_t$ is unbounded, which contradicts $(\heartsuit)$.
\end{proof}

\begin{fact}\label{suffheart}
  If $\bfrak(K)>\theta$ and $N_t\in N$ for all $t\in T$, then $(\heartsuit)$ follows.
\end{fact}

\begin{corollary}\label{cor1}
Under the assumptions of \autoref{mainlem1}, if in addition $S=K$ is a directed preorder without maximum and $T$ is a linear order, then $S\cap N\eqT T$.
\end{corollary}
\begin{proof}
Consider the functions $f$ and $g$ from the proof of \autoref{mainlem1}. Since $S$ does not have a maximum, we can even define $g$ such that $i<_S g(t)$ for all $i\in S\cap N_t$. Thus, $(g,f)\colon T\to S\cap N$ is a Tukey connection: if $g(t)\leq_S j$ in $S\cap N$, then $\forall\, i\in S\cap N_t\colon i<_S j$, so $j\notin N_t$, hence $t<_T f(j)$.  
\end{proof}

\autoref{fig:largeb1} illustrates the situation of \autoref{mainlem1} when $K=S$ is a directed preorder and $|T|\leq \theta$ (so $|\delta_N|=|N|=\theta$), while \autoref{fig:largeb2} illustrates \autoref{cor1}.

\begin{figure}[ht]
\centering
\begin{tikzpicture}[line width=1pt]
\draw[->] (-1,0) -- (10,0);
\node at (10.4,0) {$\On$};
\draw[orange] (-1,0)--(7,0);
\node at (7,-0.4) {$\colive{\delta_N}$};
\node at (7,0) {$|$};
\node at (1,0) {$|$};
\node at (1,0.4) {$\bfrak(T)$};
\node at (3,0) {$|$};
\node at (3,0.4) {$\cladri{\bfrak(S\cap N)}$};
\node at (5,0.4) {$\cladri{\dfrak(S\cap N)}$};
\node at (5,0) {$|$};
\node at (8,0.4) {$\bfrak(S)$};
\node at (8,0) {$|$};
\node at (6.5,0.4) {$\dfrak(T)$};
\node at (6.5,0) {$|$};
\end{tikzpicture}
\caption{When $K=S$ is a directed preorder and $|T|\leq \theta$, according to \autoref{mainlem1} $S\cap N\leqT T$, so the cardinal characteristics associated with $S\cap N$ lie between those associated with $T$.}\label{fig:largeb1}
\end{figure}
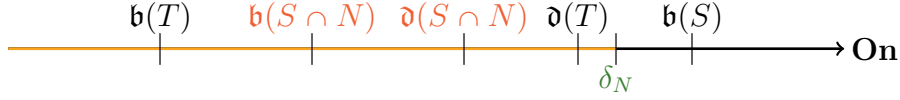

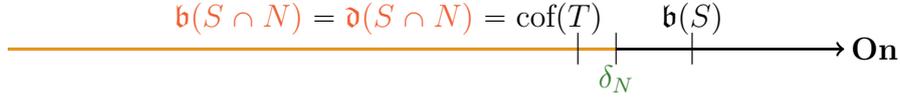
\begin{figure}[ht]
\centering
\begin{tikzpicture}[line width=1pt]
\draw[->] (-1,0) -- (10,0);
\node at (10.4,0) {$\On$};
\draw[orange] (-1,0)--(7,0);
\node at (7,-0.4) {$\colive{\delta_N}$};
\node at (7,0) {$|$};
\node at (8,0.4) {$\bfrak(S)$};
\node at (8,0) {$|$};
\node at (4,0.4) {$\cladri{\bfrak(S\cap N)}=\cladri{\dfrak(S\cap N)}=\cof(T)$};
\node at (6.5,0) {$|$};
\end{tikzpicture}
\caption{In the situation of \autoref{cor1} (when $|T|\leq\theta$), $S\cap N\eqT T$, so the cardinal characteristics associated with $S\cap N$ collapse to $\cof(T)$.}\label{fig:largeb2}
\end{figure}

We finish with a result about the intersection of a directed system of models with a chain of models.

\begin{lemma}\label{lem:2dir}
  Let $\bar N^0=\la N^k_t:\, t\in T\ra$ be a $(T,\kappa^0,\theta^0)$-directed system, and let
  $\bar N^1:=\la N^1_\alpha:\, \alpha<\lambda\ra$ be a $(\lambda,\kappa^1,\theta^1)$-directed system with $\lambda$ a limit ordinal. Assume:
  \begin{enumerate}[label=(\roman*)]
    \item $\bar N^0\in N^1_0$ and $\theta^1<\kappa^0$ (which implies $\kappa^1\leq\theta^1<\kappa^0\leq\theta^0$),
    \item $N^1_\alpha\in N^1$ for all $\alpha<\lambda$, and
    \item 
    $T\subseteq N^1_0$ (which implies $|T|\leq\theta_1$).
  \end{enumerate}
  Then:
  \begin{enumerate}[label=(\alph*)]
    \item $\bar N:=\la N_\eta:\, \eta\in\Lambda\ra$ is a $(\Lambda,\kappa^1,\theta^1)$-directed system, where
        \begin{itemize}
          \item $\Lambda:= T\times\lambda$,
          \item $N_\eta:=N^0_{\eta(0)}\cap N^1_{\eta(1)}$ for $\eta\in\Lambda$. Hence, $N=N^0\cap N^1$.
        \end{itemize}

    \item If $K=\la A,B,\lhd\ra\in N^0\cap N^1_0$ is a relational system and $\bfrak(K)>\theta^1$ then $K\cap N\leqT\Lambda$.
    In particular $\min\{\bfrak(T),\cf(\lambda)\}\leq\bfrak(K\cap N)$ and $\dfrak(K\cap N)\leq\max\{\dfrak(T),\cf(\lambda)\}$.
  \end{enumerate}
\end{lemma}
\begin{proof}
\noindent(a): Fix $\eta\in\Lambda$. Note that $N_\eta\preceq H_\chi$ because $N^0_{\eta(0)},N^1_{\eta(1)}\preceq H_\chi$ and $N^0_{\eta(0)}\in N^1_{\eta(1)}$ by~(i) and~(iii). On the other hand,
since $N^0_{\eta(0)}\in N^1_{\eta(1)}$ and $|N^0_{\eta(0)}|=\theta^0$, we get $|N^0_{\eta(0)}\cap N^1_{\eta(1)}|=|\theta^0\cap N^1_{\eta(0)}|=\theta^1$. Clearly, $N_\eta$ is ${<}\kappa^1$-closed, so we can easily conclude that $\bar N$ is a $(\Lambda,\kappa^1,\theta^1)$-directed system. Note that $N:=\bigcup_{\eta\in\Lambda}N_\eta=N^0\cap N^1$.\smallskip

\noindent(b): Let $\eta\in\Lambda$, wlog $K\in N^0_{\eta(0)}$ (by increasing $\eta(0)$). Since $\bfrak(K\cap N^0_{\eta(0)})\geq\min\{\bfrak(K),\kappa_0\}>\theta^1$ by \autoref{fct:ScapN}~(b) and (i), and $A\cap N_\eta\in N^1$ (by (ii)) has size ${\leq}\theta^1$,
\[N^1\models \exists\, y\in B\cap N^0_{\eta(0)}\ \forall\, x\in A\cap N_\eta\colon x\lhd y,\]
so we can pick such $y\in B\cap N^0_{\eta(0)}\cap N^1\subseteq N$.
Hence $(\heartsuit)$ of \autoref{mainlem1} holds, thus $K\cap N\leqT\Lambda$.
\end{proof}

\section{Cicho\'n's maximum}\label{sec:cmax}

We fix cardinals ordered as in \autoref{fig:cichoncoll}, all of them regular with the possible exception of $\lambda^\cfrak$. Applying \autoref{gksmax}, we first construct a ccc poset $\Por$ forcing $\cfrak=\theta_\infty$ and the constellation at the top with Tukey connections, namely, $\Rbf_i\eqT S_i$ for all $1\leq i\leq 4$, where $S_i:=[\theta_\infty]^{<\theta_i}\cap V$ is a directed partial order.

\begin{figure}[ht]
\centering
\begin{tikzpicture}[xscale=2/1]
\footnotesize{
\node (addn) at (0,3.5){$\theta_1$};
\node (covn) at (0,5.5){$\theta_2$};
\node (nonn) at (3,3.5) {$\theta_\infty$} ;
\node (cfn) at (3,5.5) {$\theta_\infty$} ;
\node (addm) at (1,3.5) {$\bullet$} ;
\node (covm) at (2,3.5) {$\theta_\infty$} ;
\node (nonm) at (1,5.5) {$\theta_4$} ;
\node (cfm) at (2,5.5) {$\bullet$} ;
\node (b) at (1,4.5) {$\theta_3$};
\node (d) at (2,4.5) {$\theta_\infty$};
\node (c) at (4,5.5) {$\theta_\infty$};

\draw (addn) edge[->] node[left] {$\theta^-_2$} (covn);
\draw[gray]
      (covn) edge [->] (nonm)
      (nonm)edge [->] (cfm)
      (cfm)edge [->] (cfn)
      (addn) edge [->]  (addm)
      (addm) edge [->]  (covm)
      (addm) edge [->] (b)
      (d)  edge[->] (cfm)
      (b) edge [->] (d)
      (nonn) edge [->]  (cfn)
      (cfn) edge[->] (c)
      (covm) edge [->] (d)
      (covm) edge [->]  (nonn);

\draw (b)  edge [->] node[left] {$\theta^-_4$} (nonm);

\draw[dashed] (covn) edge[->] node[below] {$\theta^-_3$} (b);
\draw[dashed] (nonm) edge[->] (covm);

\node (aleph1) at (-1,0) {$\aleph_1$};
\node (addn-f) at (0,0){$\lambda_1^\bfrak$};
\node (covn-f) at (0,2){$\lambda_2^\bfrak$};
\node (nonn-f) at (3,0) {$\lambda_2^\dfrak$} ;
\node (cfn-f) at (3,2) {$\lambda_1^\dfrak$} ;
\node (addm-f) at (1,0) {$\bullet$} ;
\node (covm-f) at (2,0) {$\lambda_4^\dfrak$} ;
\node (nonm-f) at (1,2) {$\lambda_4^\bfrak$} ;
\node (cfm-f) at (2,2) {$\bullet$} ;
\node (b-f) at (1,1) {$\lambda_3^\bfrak$};
\node (d-f) at (2,1) {$\lambda_3^\dfrak$};
\node (c-f) at (4,2) {$\lambda^{\cfrak}$};

\draw[dashed] (c-f) edge[->] node[above] {$\theta^-_1$} (addn);
\draw (aleph1) edge[->] (addn-f)
      (addn-f) edge[->] (covn-f)
      (nonn-f) edge [->]  (cfn-f)
      (b-f)  edge [->] (nonm-f)
      (covm-f) edge [->] (d-f)
      (cfn-f) edge[->] (c-f);

\draw[gray]
   (covn-f) edge [->] (nonm-f)
   (nonm-f)edge [->] (cfm-f)
   (cfm-f)edge [->] (cfn-f)
   (addn-f) edge [->]  (addm-f)
   (addm-f) edge [->]  (covm-f)
   (covm-f) edge [->]  (nonn-f)
   (addm-f) edge [->] (b-f)
   (d-f)  edge[->] (cfm-f)
   (b-f) edge [->] (d-f);

\draw[dashed] (covn-f) edge [->] (b-f)
              (nonm-f) edge [->] (covm-f)
              (d-f) edge [->] (nonn-f);

}
\end{tikzpicture}
\caption{Strategy to force Cicho\'n's maximum: we construct a ccc poset $\Por$ forcing the constellation at the top, and find a $\sigma$-closed model $N$ such that $\Por\cap N$ forces the constellation at the bottom.}\label{fig:cichoncoll}
\end{figure}
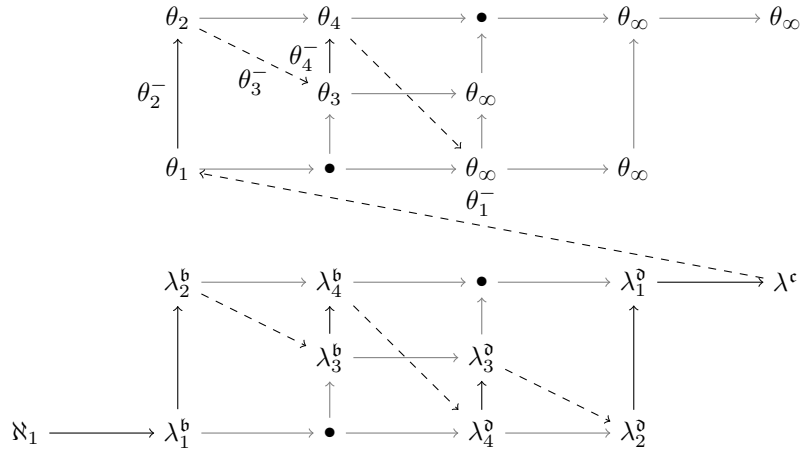

Afterwards, we apply the theory of \autoref{sec:restrmod} to construct a $\sigma$-closed $N\preceq H_\chi$, where $\chi$ is chosen regular large enough, such that $\Por\cap N$ forces the Cicho\'n's maximum constellation at the bottom. By \autoref{fct:restrN} we obtain that $\Por\cap N$ forces $\Rbf_i\eqT S_i\cap N$, so we need to construct $N$ such that $\bfrak(S_i\cap N)=\lambda_i^\bfrak$ and $\dfrak(S_i\cap N)=\lambda_i^\dfrak$. We will have $|N|=\lambda^\cfrak$, so $\Por\cap N$ will force that $\cfrak=\lambda^\cfrak$.

The strategy is to construct several chains of elementary submodels of $H_\chi$ and intersect them.
To proceed, we fix the following assumptions and conventions:

\begin{enumerate}[label=(H\arabic*)]
  \item\label{hyp1} Cardinals ordered as in \autoref{fig:cichoncoll}, non-decreasing up to $\lambda^\cfrak$ and increasing from there.
  \item\label{hyp2} With the possible exception of $\lambda^\cfrak$, all cardinals are regular. But we assume $(\lambda^\cfrak)^{\aleph_0}=\lambda^\cfrak$.
  \item\label{hyp3} The cardinals $\theta_i$ ($1\leq i\leq 4$) and $\theta_\infty$ satisfy the hypothesis of \autoref{gksmax}.
  \item\label{hyp4} For every $1\leq i\leq 4$, $\theta_i^{\theta^-_i}=\theta_i$ and $(\theta_i^-)^{<\theta^-_i}=\theta^-_i$.
  \item\label{hyp5} All models from now on contain as elements all the cardinals in \ref{hyp1}.
  \item\label{hyp6} Every new model contains, as elements, all the chains of models previously defined.
\end{enumerate}

More concretely, we prove:

\begin{theorem}\label{thmcmax}
Under assumptions \ref{hyp1}--\ref{hyp4}, there is a ccc poset forcing, for $1\leq i\leq 4$:
\begin{enumerate}[label=(\alph*)]
    \item $\Rbf_i\leqT\prod_{j=i}^4\lambda^\dfrak_j\times\lambda^\bfrak_j$,
    \item $\lambda^\bfrak_j\leqT\Rbf_i$ and $\lambda^\dfrak_j\leqT\Rbf_i$ when $i\leq j\leq 4$, and
    \item $\cfrak=\lambda^\cfrak$.
\end{enumerate}
\end{theorem}

We present two proofs. The first one is a short compact proof, and the second is the argument step by step, showing how cardinal characteristics are modified.

\begin{proof}[Proof (compact version)] 
By~\ref{hyp3}, find a ccc poset $\Por$ as in \autoref{gksmax}, i.e.\ forcing $\cfrak=\theta_\infty$ and $\R_i\eqT S_i$ for all $1\leq i\leq 4$. On the other hand, we have Tukey relations of regular cardinals with $S_i$ and the values of its associated cardinal characteristics as in \autoref{tbl0}.

\begin{table}[ht]
\centering
\begin{tabular}{c||c|c|c}
     & (regular) & & \\
 $i$ & $\theta\leqT S_i$ & $\bfrak(S_i)$ & $\dfrak(S_i)$ \\
\hline\hline

$4$ & $[\theta_4,\theta_\infty]$ & $\theta_4$ & $\theta_\infty$ \\

\hline

$3$ & $[\theta_3,\theta_\infty]$ &$ \theta_3$ & $\theta_\infty$\\

\hline

$2$ & $[\theta_2,\theta_\infty]$ & $\theta_2$ & $\theta_\infty$\\

\hline

$1$ & $[\theta_1,\theta_\infty]$ & $\theta_1$ & $\theta_\infty$\\

\end{tabular}
\caption{Values of the cardinal characteristics of $S_i$ and Tukey connections from regular cardinals (consequence of \autoref{impEUB}).}\label{tbl0}
\end{table}

By downwards recursion on $1\leq i\leq 4$, we construct chains of models $\bar{N}^\dfrak_i:=\la N^\dfrak_{i,\alpha}:\, \alpha<\lambda^\dfrak_i\ra$ and $\bar{N}^\bfrak_i:=\la N^\bfrak_{i,\alpha}:\, \alpha<\lambda^\bfrak_i\ra$ satisfying:
\begin{enumerate}[label=(\roman*)]
    \item $\bar{N}^\dfrak_i$ is a $(\lambda^\dfrak_i,(\theta^-_i)^+,\theta_i)$-directed system;
    \item $\bar{N}^\bfrak_i$ is a $(\lambda^\bfrak_i,\theta^-_i,\theta^-_i)$-directed system;
    \item each $N^\dfrak_{i,\alpha}$ contains, as elements: the cardinals of \autoref{fig:cichoncoll},  $\Por$, $\la N^\dfrak_{i,\xi}:\, \xi<\alpha\ra$, and the sequences $\bar N^\dfrak_j$ and $\bar N^\bfrak_j$ for all $i<j\leq4$;
    \item each $N^\bfrak_{i,\alpha}$ contains, as elements: the cardinals of \autoref{fig:cichoncoll},  $\Por$, $\la N^\bfrak_{i,\xi}:\, \xi<\alpha\ra$, and the sequences $\bar{N}^\dfrak_i$, $\bar N^\dfrak_j$ and $\bar N^\bfrak_j$ for all $i<j\leq4$.
\end{enumerate}
Assumptions \ref{hyp2} and~\ref{hyp4} are what allow the construction of the models. Note that~(iii) and~(iv) obey~\ref{hyp5} and~\ref{hyp6}, and these imply that the models contain, as elements, $S_i$ and the parameters of the definition of $\Rbf_i$ for all $1\leq i\leq 4$.

Finally, let $N^\cfrak\preceq H_\chi$ be a $\sigma$-closed model of size $\lambda^\cfrak$ containing, as elements, everything we have so far (this is possible because $(\lambda^\cfrak)^{\aleph_0}=\lambda^\cfrak$, see~\ref{hyp2}). We let $N:=N^\cfrak\cap\bigcap_{j=1}^4 N^\dfrak_j\cap N^\bfrak_j$, and show that $\Por\cap N$ is as desired. To show (a) and (b), note that $\Por\cap N$ forces $\Rbf_i\eqT S_i\cap N$ for all $1\leq i\leq 4$, hence it is enough to show that
\begin{enumerate}[label=(\alph*')]
   \item $S_i\cap N\leqT\Lambda_i:=\prod_{j=i}^4\lambda^\dfrak_j\times\lambda^\bfrak_j$, and
   \item $\lambda^\bfrak_j\leqT S_i\cap N$ and $\lambda^\dfrak_j\leqT S_i\cap N$ when $i\leq j\leq 4$.
\end{enumerate}
Item (c) follows because $|\Por\cap N|=\lambda^\cfrak$.\smallskip

\noindent(a'): Let $\Lambda'_i:=\lambda^\dfrak_i\times\prod_{j=i+1}^4 \lambda^\dfrak_j\times\lambda^\bfrak_j$ (when $i=4$, just let $\Lambda'_4:=\lambda^\dfrak_4$), which is a relational system. As in the proof of \autoref{lem:2dir}~(a), the intersection of the chain of models $\bar{N}^\dfrak_i$ with $\bar N^\dfrak_j$ and $\bar N^\bfrak_j$ for all $i<j\leq4$ yields a $(\Lambda'_i,(\theta^-_i)^+,\theta_i)$ directed system $\bar N'_i$. Since $\bfrak(S_i)=\theta_i>\theta^-_i$, by \autoref{lem:2dir} applied to $\bar N^0=\bar N'_i$ and $\bar N^1=\bar N^\bfrak_i$, we obtain a $(\Lambda_i,\theta^-_i,\theta^-_i)$-directed system $\bar N_i$ such that $S_i\cap N_i\leqT \Lambda_i$ where $N_i=\bigcap_{j=i}^4 N^\dfrak_i\cap N^\bfrak_i$.

Now, by regressive induction on $1\leq j<i$, we show that $S_i\cap N_j\eqT S_i\cap N_i$. Assume we have the result for $j+1$ (which we showed for $j+1=i$). Since $\dfrak(S_i\cap N_{j+1})\leq\dfrak(\Lambda_i)=\lambda^\dfrak_i\subseteq N^\dfrak_j$, by \autoref{fct:ScapN}~(d) we obtain $S_i\cap N_{j+1}\cap N^\dfrak_j\eqT S_i\cap N_{j+1}\eqT S_i\cap N_i$. For the same reason, we get $S_i\cap N_j=S_i\cap N_{j+1}\cap N^\dfrak_j\cap N^\bfrak_j\eqT S_i\cap N_{j+1}\cap N^\dfrak_j\eqT S_i\cap N_i$.

Finally, by applying \autoref{fct:ScapN}~(d), we conclude $S_i\cap N= S_i\cap N_1\cap N^\cfrak\eqT S_i\cap N_1\eqT S_i\cap N_i\leqT \Lambda_i$.\smallskip

\noindent(b'): For $1\leq j\leq 4$ we consider $\bar N'_j$ and $\bar N_j$ as defined in the previous argument. Fix $i\leq j\leq 4$. We have that $\theta_j\leqT S_i$ and $\theta^-_{j+1}\leqT S_i$ (denote $\theta^-_5:=\theta_\infty$), which imply $\theta_j\cap N\leqT S_i\cap N$ and $\theta^-_{j+1}\cap N\leqT S_i\cap N$.
So it is enough to show that $\theta_j\cap N\eqT\lambda^\bfrak_i$ and $\theta^-_{j+1}\cap N\eqT\lambda^\dfrak_j$.

Since $\theta_\infty>\theta_4=|N^\dfrak_4|$, by \autoref{cor1} applied to $\bar N^\dfrak_4$ we get $\theta_\infty\cap N^\dfrak_4\eqT\lambda^\dfrak_4$, showing $\theta^-_{j+1}\cap N'_j\eqT\lambda^\dfrak_j$ for $j=4$; in the case $j<4$, since $\theta^-_{j+1}=|N_{j+1}|$, we get $\theta^-_{j+1}\cap N_{j+1}\eqT\theta^-_{j+1}$ by \autoref{fct:ScapN}~(d) (even equality holds), but $|N^\dfrak_j|=\theta_j<\theta^-_{j+1}$, so \autoref{cor1} applied to $\bar N^\dfrak_j$ implies $\theta^-_{j+1}\cap N'_j\eqT \lambda^\dfrak_j$.

Back to $j\leq 4$: since $\theta_j=|N'_j|$ we have $\theta_j\cap N'_j\eqT\theta_j$ by \autoref{fct:ScapN}~(d) (even equality holds). Now $\theta_j>\theta_j^-=|N^\bfrak_j|$, so by \autoref{cor1} applied to $\bar N^\bfrak_j$ we obtain $\theta_j\cap N_j\eqT \lambda^\bfrak_j$. On the other hand, $\theta^-_{j+1}\cap N_j\eqT \theta^-_{j+1}\cap N'_j\eqT \lambda^\dfrak_j$ by \autoref{fct:ScapN}~(d). Now, using \autoref{fct:ScapN}~(d), it is easy to show by decreasing recursion on $1\leq k<j$ that $\theta_j\cap N_k\eqT\lambda^\bfrak_j$ and $\theta^-_{j+1}\cap N_k\eqT \lambda^\dfrak_j$. For the same reason, $\theta_j\cap N=\theta_j\cap N_1\cap N^\cfrak\eqT \theta_j\cap N_1\eqT\lambda^\bfrak_j$ and $\theta^-_{j+1}\cap N\eqT\lambda^\dfrak_j$.
\end{proof}

We now explain what occurs step by step when intersecting with the chain of models in the previous proof. By \autoref{gksmax} we obtain a ccc poset $\Por$ that forces $\R_i\eqT S_i$ for all $1\leq i\leq 4$. Recall \autoref{tbl0} about the values obtained for $S_i$.

\medskip

\noindent \textbf{Step~1.1.}
Construct a $(\lambda^\dfrak_4,(\theta^-_4)^+,\theta_4)$-directed system $\bar{N}^\dfrak_4:=\la N^\dfrak_{4,\alpha}:\, \alpha<\lambda^\dfrak_4\ra$ such that $N^\dfrak_{4,\alpha}\in N^\dfrak_{4,\alpha+1}$ (using $\theta_4^{\theta^-_4}=\theta_4$). Thus $N^\dfrak_4$ is ${<}\lambda^\dfrak_4$-closed and
$\Por^\dfrak_4:=\Por\cap N^\dfrak_4$ forces $\Rbf_i\eqT S^{4,\dfrak}_i:=S_4\cap N^\dfrak_4$ (by \autoref{fct:restrN}). So the values forced to $\bfrak(\Rbf_i)$ and $\dfrak(\Rbf_i)$ are according to \autoref{tbl1.1}.

\begin{table}[ht]
\centering
\begin{tabular}{c||c|c|c}
     & (regular) & & \\
 $i$ & $\theta\leqT S^{4,\dfrak}_i$ & $\bfrak(S_i^{4,\dfrak})$ & $\dfrak(S_i^{4,\dfrak})$ \\
\hline\hline

$4$ &$\theta_4,\blue{\lambda^\dfrak_4}$ & \blue{$\lambda^\dfrak_4$} & $\cladri{\theta_4}$ \\

\hline

$3$ & $[\theta_3,\cladri{\theta_4}],\blue{\lambda^\dfrak_4}$ & $\blue{\lambda^\dfrak_4}$ & $\cladri{\theta_4}$\\

\hline

$2$ & $[\theta_2,\cladri{\theta_4}],\blue{\lambda^\dfrak_4}$ & $\blue{\lambda^\dfrak_4}$ & $\cladri{\theta_4}$\\

\hline

$1$ & $[\theta_1,\cladri{\theta_4}],\blue{\lambda^\dfrak_4}$ & $\blue{\lambda^\dfrak_4}$ & $\cladri{\theta_4}$\\



\end{tabular}
\caption{Values of the cardinal characteristics of $S^{4,\dfrak}_i$ and Tukey connections from regular cardinals.}\label{tbl1.1}
\end{table}

To prove the values in \autoref{tbl1.1}, by \autoref{fct:ScapN}~(b) note that $\bfrak(S_i^{4,\dfrak})\geq\min\{\bfrak(S_i),\lambda^\dfrak_4\}=\lambda^\dfrak_4$. On the other hand, since $\theta_\infty>\theta_4$, by \autoref{cor1} we obtain $\theta_\infty\cap N^\dfrak_4\eqT\lambda^\dfrak_4$, so $\theta_\infty\leqT S_i$ implies $\lambda^\dfrak_4\eqT \theta_\infty\cap N^\dfrak_4\leqT S^{4,\dfrak}_i$. Therefore, $\bfrak(S^{4,\dfrak}_i)=\lambda^\dfrak_4$.

By \autoref{fct:ScapN}~(d), for any regular $\theta\leq\theta_4=|N^\dfrak_4|$, we obtain $\theta\cap N^\dfrak_4\eqT\theta$, so $\theta\leqT S_i$ implies $\theta\leqT S_i\cap N^\dfrak_4=S^{4,\dfrak}_i$. In particular, for $\theta=\theta_4$, we obtain $\theta_4\leq\dfrak(S^{4,\dfrak}_i)$. The converse inequality holds by \autoref{fct:ScapN}~(a) as $\dfrak(S_i\cap N^\dfrak_4)\leq|\dfrak(S_i)\cap N^\dfrak_4|=\theta_4$.\medskip

\noindent\textbf{Step~1.2.} Construct a $(\lambda^\bfrak_4,\theta^-_4,\theta^-_4)$-directed system $\bar{N}^\bfrak_4:=\la N^\bfrak_{4,\alpha}:\, \alpha<\lambda^\bfrak_4\ra$ such that $N^\bfrak_{4,\alpha}\in N^\bfrak_{4,\alpha+1}$ (using $(\theta^-_4)^{<\theta^-_4}=\theta^-_4$). Thus, $N^\bfrak_4$ is ${<}\lambda^\bfrak_4$-closed and
$\Por^\bfrak_4:=\Por^\dfrak_4\cap N^\bfrak_4$ forces $\Rbf_i\eqT S_i^{4,\bfrak}:= S_i^{4,\dfrak}\cap N^\bfrak_4$. The values of the cardinals of $S_i^{4,\bfrak}$ are displayed in \autoref{tbl1.2}.

\begin{table}[ht]
\centering
\begin{tabular}{c||c|c|c}
     & (regular) & & \\
 $i$ & $\theta\leqT S^{4,\bfrak}_i$ & $\bfrak(S_i^{4,\bfrak})$ & $\dfrak(S_i^{4,\bfrak})$ \\
\hline\hline

$4$ & $\cyan{\lambda^\bfrak_4},\blue{\lambda^\dfrak_4}$ & $\cyan{\lambda^\bfrak_4}$ & $\azul{\lambda^\dfrak_4}$ \\

\hline

$3$ & $[\theta_3,\cladri{\theta_4^-}],\cyan{\lambda^\bfrak_4},\blue{\lambda^\dfrak_4}$ & $\cyan{\lambda^\bfrak_4}$ & $\cladri{\theta_4^-}$\\

\hline

$2$ & $[\theta_2,\cladri{\theta_4^-}],\cyan{\lambda^\bfrak_4},\blue{\lambda^\dfrak_4}$ & $\cyan{\lambda^\bfrak_4}$ & $\cladri{\theta_4^-}$\\

\hline

$1$ & $[\theta_1,\cladri{\theta_4^-}],\cyan{\lambda^\bfrak_4},\blue{\lambda^\dfrak_4}$ & $\cyan{\lambda^\bfrak_4}$ & $\cladri{\theta_4^-}$\\



\end{tabular}
\caption{Values of the cardinal characteristics of $S^{4,\bfrak}_i$ and Tukey connections from regular cardinals.}\label{tbl1.2}
\end{table}

Note that $\bfrak(S_i^{4,\bfrak})\geq\min\{\bfrak(S_i^{4,\dfrak}),\lambda^\bfrak_4\}=\lambda^\bfrak_4$ by \autoref{fct:ScapN}~(b). On the other hand, since $\theta_4>\theta^-_4$, by \autoref{cor1} we obtain $\theta_4\cap N^\bfrak_4\eqT\lambda^\bfrak_4$, so $\theta_4\leqT S^{4,\dfrak}_i$ implies $\lambda^\bfrak_4\eqT \theta_4\cap N^\bfrak_4\leqT S^{4,\bfrak}_i$. Therefore, $\bfrak(S_i^{4,\bfrak})=\lambda^\bfrak_4$.

By \autoref{fct:ScapN}~(d), for any regular $\theta\leq\theta^-_4=|N^\bfrak_4|$, we obtain $\theta\cap N^\bfrak_4\eqT\theta$, so $\theta\leqT S_i^{4,\dfrak}$ implies $\theta\leqT S_i^{4,\dfrak}\cap N^\bfrak_4=S^{4,\bfrak}_i$. In particular, for $\theta=\lambda^\dfrak_4$ we obtain $\lambda^\dfrak_4\leq\dfrak(S^{4,\bfrak}_i)$, and for $i\leq 3$ and
$\theta=\theta^-_4$, we obtain $\theta^-_4\leq\dfrak(S^{4,\bfrak}_i)$. The converse inequality holds by \autoref{fct:ScapN}~(a) as $\dfrak(S_i^{4,\dfrak}\cap N^\bfrak_4)\leq|\dfrak(S_i^{4,\dfrak})\cap N^\bfrak_4|=\theta^-_4$.

It remains to show that $\dfrak(S_4^{4,\dfrak})\leq\lambda^\dfrak_4$. Note that the hypothesis of \autoref{lem:2dir} holds for $\bar{N}^0=\bar N^\dfrak_4$ and $\bar N^1=\bar N^\bfrak_4$, so, since $\bfrak(S_4)=\theta_4>\theta^-_4$, we conclude that $S_4^{4,\bfrak}=S_4\cap N^\dfrak_4\cap N^\bfrak_4\leqT \lambda^\dfrak_4\times\lambda^\bfrak_4$. Therefore $\lambda^\bfrak_4\leq\bfrak(S_4^{4,\bfrak})\leq\dfrak(S_4^{4,\bfrak})\leq\lambda^\dfrak_4$.\medskip

\noindent\textbf{Step~2.1.} Construct a $(\lambda^\dfrak_3,(\theta^-_3)^+,\theta_3)$-directed system $\bar{N}^\dfrak_3:=\la N^\dfrak_{3,\alpha}:\, \alpha<\lambda^\dfrak_3\ra$ such that $N^\dfrak_{3,\alpha}\in N^\dfrak_{3,\alpha+1}$ (using $\theta_3^{\theta^-_3}=\theta_3$). Thus $N^\dfrak_3$ is ${<}\lambda^\dfrak_3$-closed and $\Por^\dfrak_3:=\Por^\bfrak_4\cap N^\dfrak_3$ forces $\Rbf_i\eqT S_i^{3,\dfrak}:= S_i^{4,\bfrak}\cap N^\dfrak_3$. The values of the cardinals of $S_i^{3,\dfrak}$ are displayed in \autoref{tbl2.1}.

\begin{table}[ht]
\centering
\begin{tabular}{c||c|c|c}
     & (regular) & & \\
 $i$ & $\theta\leqT S^{3,\dfrak}_i$ & $\bfrak(S_i^{3,\dfrak})$ & $\dfrak(S_i^{3,\dfrak})$ \\
\hline\hline

$4$ & $\cyan{\lambda^\bfrak_4},\blue{\lambda^\dfrak_4}$ & $\cyan{\lambda^\bfrak_4}$ & $\azul{\lambda^\dfrak_4}$ \\

\hline

$3$ & $\cladri{\theta_3},\cyan{\lambda^\bfrak_4}, \blue{\lambda^\dfrak_4},\colive{\lambda^\dfrak_3}$ & $\cyan{\lambda^\bfrak_4}$ & $\cladri{\theta_3}$\\

\hline

$2$ & $[\theta_2,\cladri{\theta_3}],\cyan{\lambda^\bfrak_4},\blue{\lambda^\dfrak_4}, \colive{\lambda^\dfrak_3}$ & $\cyan{\lambda^\bfrak_4}$ & $\cladri{\theta_3}$\\

\hline

$1$ & $[\theta_1,\cladri{\theta_3}],\cyan{\lambda^\bfrak_4},\blue{\lambda^\dfrak_4}, \colive{\lambda^\dfrak_3}$ & $\cyan{\lambda^\bfrak_4}$ & $\cladri{\theta_3}$\\



\end{tabular}
\caption{Values of the cardinal characteristics of $S^{3,\dfrak}_i$ and Tukey connections from regular cardinals.}\label{tbl2.1}
\end{table}

The values for $1\leq i\leq3$ are calculated similarly to Steps~1.1 and~1.2, so we only explain the values for $i=4$. Since $\dfrak(S^{4,\bfrak}_4)=\lambda^\dfrak_4<\theta_3=|N^\dfrak_3|$, by \autoref{fct:ScapN}~(d) we obtain $S^{3,\dfrak}_4\eqT S^{4,\bfrak}_4$, so the values of the cardinal characteristics stay the same.
\medskip

\noindent\textbf{Step~2.2.} Construct a $(\lambda^\bfrak_3,\theta^-_3,\theta^-_3)$-directed system $\bar{N}^\bfrak_3:=\la N^\bfrak_{3,\alpha}:\, \alpha<\lambda^\bfrak_3\ra$ such that $N^\bfrak_{3,\alpha}\in N^\bfrak_{3,\alpha+1}$ (using $(\theta_3^-)^{<\theta^-_3}=\theta_3^-$). Thus $N^\bfrak_3$ is ${<}\lambda^\bfrak_3$-closed and $\Por^\bfrak_3:=\Por^\dfrak_3\cap N^\bfrak_3$ forces $\Rbf_i\eqT S_i^{3,\bfrak}:= S_i^{3,\dfrak}\cap N^\bfrak_3$. The values of the cardinals of $S_i^{3,\bfrak}$ are displayed in \autoref{tbl2.2}. 

\begin{table}[ht]
\centering
\begin{tabular}{c||c|c|c}
     & (regular) & & \\
 $i$ & $\theta\leqT S_i^{3,\bfrak}$ & $\bfrak(S_i^{3,\bfrak})$ & $\dfrak(S_i^{3,\bfrak})$ \\
\hline\hline

$4$ & $\cyan{\lambda^\bfrak_4},\blue{\lambda^\dfrak_4}$ & $\cyan{\lambda^\bfrak_4}$ & $\azul{\lambda^\dfrak_4}$ \\

\hline

$3$ & $\green{\lambda^\bfrak_3}, \cyan{\lambda^\bfrak_4}, \blue{\lambda^\dfrak_4},\colive{\lambda^\dfrak_3}$ & $\green{\lambda^\bfrak_3}$ & $\colive{\lambda^\dfrak_3}$\\

\hline

$2$ & $[\theta_2,\cladri{\theta_3^-}],\green{\lambda^\bfrak_3}, \cyan{\lambda^\bfrak_4},\blue{\lambda^\dfrak_4}, \colive{\lambda^\dfrak_3}$ & $\green{\lambda^\bfrak_3}$ & $\cladri{\theta_3^-}$\\

\hline

$1$ & $[\theta_1,\cladri{\theta_3^-}],\green{\lambda^\bfrak_3}, \cyan{\lambda^\bfrak_4},\blue{\lambda^\dfrak_4}, \colive{\lambda^\dfrak_3}$ & $\green{\lambda^\bfrak_3}$ & $\cladri{\theta_3^-}$\\



\end{tabular}
\caption{Values of the cardinal characteristics of $S^{3,\bfrak}_i$ and Tukey connections from regular cardinals.}\label{tbl2.2}
\end{table}

This is similar to Steps~1.2 and~2.1, but $\dfrak(S_3^{3,\bfrak})\leq\lambda^\dfrak_3$ needs more details. As in the proof of \autoref{lem:2dir}, $N^\dfrak_4\cap N^\bfrak_4\cap N^\dfrak_3$ is obtained by a $(\lambda^\dfrak_3\times\lambda^\bfrak_4\times\lambda^\dfrak_4,(\theta^-_3)^+,\theta_3)$-directed system $\bar{N}'_3$. So we apply \autoref{lem:2dir} to $\bar N^0=\bar{N}'_3$ and $\bar N^1=\bar N^\bfrak_3$ to obtain $S^{3,\bfrak}_3\leqT \prod_{j=3}^4\lambda^\dfrak_j\times\lambda^\bfrak_j$.\medskip

We proceed in the same fashion for the remaining steps.\medskip

\noindent\textbf{Step~3.1} Construct a $(\lambda^\dfrak_2,(\theta^-_2)^+,\theta_2)$-directed system $\bar{N}^\dfrak_2:=\la N^\dfrak_{2,\alpha}:\, \alpha<\lambda^\dfrak_2\ra$ such that $N^\dfrak_{2,\alpha}\in N^\dfrak_{2,\alpha+1}$ (using $\theta_2^{\theta^-_2}=\theta_2$). Thus $N^\dfrak_2$ is ${<}\lambda^\dfrak_2$-closed and $\Por^\dfrak_2:=\Por^\bfrak_3\cap N^\dfrak_2$ forces $\Rbf_i\eqT S_i^{2,\dfrak}:= S_i^{3,\bfrak}\cap N^\dfrak_2$. The values of the cardinals of $S_i^{2,\dfrak}$ are displayed in \autoref{tbl3.1}.

\begin{table}[ht]
\centering
\begin{tabular}{c||c|c|c}
     & (regular) & & \\
 $i$ & $\theta\leqT S_i^{2,\dfrak}$ & $\bfrak(S_i^{2,\dfrak})$ & $\dfrak(S_i^{2,\dfrak})$ \\
\hline\hline

$4$ & $\cyan{\lambda^\bfrak_4},\blue{\lambda^\dfrak_4}$ & $\cyan{\lambda^\bfrak_4}$ & $\azul{\lambda^\dfrak_4}$ \\

\hline

$3$ & $\green{\lambda^\bfrak_3}, \cyan{\lambda^\bfrak_4}, \blue{\lambda^\dfrak_4},\colive{\lambda^\dfrak_3}$ & $\green{\lambda^\bfrak_3}$ & $\colive{\lambda^\dfrak_3}$\\

\hline

$2$ & $\cladri{\theta_2},\green{\lambda^\bfrak_3}, \cyan{\lambda^\bfrak_4},\blue{\lambda^\dfrak_4}, \colive{\lambda^\dfrak_3}, \ccafe{\lambda^\dfrak_2}$ & $\green{\lambda^\bfrak_3}$ & $\cladri{\theta_2}$\\

\hline

$1$ & $[\theta_1,\cladri{\theta_2}],\green{\lambda^\bfrak_3}, \cyan{\lambda^\bfrak_4},\blue{\lambda^\dfrak_4}, \colive{\lambda^\dfrak_3}, \ccafe{\lambda^\dfrak_2}$ & $\green{\lambda^\bfrak_3}$ & $\cladri{\theta_2}$\\



\end{tabular}
\caption{Values of the cardinal characteristics of $S^{2,\dfrak}_i$ and Tukey connections from regular cardinals.}\label{tbl3.1}
\end{table}
\medskip

\noindent\textbf{Step~3.2.} Construct a $(\lambda^\bfrak_2,\theta^-_2,\theta^-_2)$-directed system $\bar{N}^\bfrak_2:=\la N^\bfrak_{2,\alpha}:\, \alpha<\lambda^\bfrak_2\ra$ such that $N^\bfrak_{2,\alpha}\in N^\bfrak_{2,\alpha+1}$ (using $(\theta_2^-)^{<\theta^-_2}=\theta_2^-$). Thus $N^\bfrak_2$ is ${<}\lambda^\bfrak_2$-closed and $\Por^\bfrak_2:=\Por^\dfrak_2\cap N^\bfrak_2$ forces $\Rbf_i\eqT S_i^{2,\bfrak}:= S_i^{2,\dfrak}\cap N^\bfrak_2$. The values of the cardinals of $S_i^{2,\bfrak}$ are displayed in \autoref{tbl3.2}, in particular, we obtain $S_2^{2,\bfrak}\leqT\prod_{j=2}^4 \lambda^\dfrak_j\times\lambda^\bfrak_j$.

\begin{table}[ht]
\centering
\begin{tabular}{c||c|c|c}
     & (regular) & & \\
 $i$ & $\theta\leqT \Rbf_i$ & $\bfrak(S_i^{2,\bfrak})$ & $\dfrak(S_i^{2,\bfrak})$ \\
\hline\hline

$4$ & $\cyan{\lambda^\bfrak_4},\blue{\lambda^\dfrak_4}$ & $\cyan{\lambda^\bfrak_4}$ & $\azul{\lambda^\dfrak_4}$ \\

\hline

$3$ & $\green{\lambda^\bfrak_3}, \cyan{\lambda^\bfrak_4}, \blue{\lambda^\dfrak_4},\colive{\lambda^\dfrak_3}$ & $\green{\lambda^\bfrak_3}$ & $\colive{\lambda^\dfrak_3}$\\

\hline

$2$ & $\ccanela{\lambda^\bfrak_2},\green{\lambda^\bfrak_3}, \cyan{\lambda^\bfrak_4},\blue{\lambda^\dfrak_4}, \colive{\lambda^\dfrak_3}, \ccafe{\lambda^\dfrak_2}$ & $\ccanela{\lambda^\bfrak_2}$ & $\ccafe{\lambda^\dfrak_2}$\\

\hline

$1$ & $[\theta_1,\cladri{\theta^-_2}],\ccanela{\lambda^\bfrak_2},\green{\lambda^\bfrak_3}, \cyan{\lambda^\bfrak_4},\blue{\lambda^\dfrak_4}, \colive{\lambda^\dfrak_3}, \ccafe{\lambda^\dfrak_2}$ & $\ccanela{\lambda^\bfrak_2}$ & $\cladri{\theta^-_2}$\\



\end{tabular}
\caption{Values of the cardinal characteristics of $S^{2,\bfrak}_i$ and Tukey connections from regular cardinals.}\label{tbl3.2}
\end{table}
\medskip

\noindent\textbf{Step~4.1.} Construct a $(\lambda^\dfrak_1,(\theta^-_1)^+,\theta_1)$-directed system $\bar{N}^\dfrak_1:=\la N^\dfrak_{1,\alpha}:\, \alpha<\lambda^\dfrak_1\ra$ such that $N^\dfrak_{1,\alpha}\in N^\dfrak_{1,\alpha+1}$ (using $\theta_1^{\theta^-_1}=\theta_1$). Thus $N^\dfrak_1$ is ${<}\lambda^\dfrak_1$-closed and $\Por^\dfrak_1:=\Por^\bfrak_2\cap N^\dfrak_1$ forces $\Rbf_i\eqT S_i^{1,\dfrak}:= S_i^{2,\bfrak}\cap N^\dfrak_1$. The values of the cardinals of $S_i^{1,\dfrak}$ are displayed in \autoref{tbl4.1}.

\begin{table}[ht]
\centering
\begin{tabular}{c||c|c|c}
     & (regular) & & \\
 $i$ & $\theta\leqT \Rbf_i$ & $\bfrak(S_i^{1,\dfrak})$ & $\dfrak(S_i^{1,\dfrak})$ \\
\hline\hline

$4$ & $\cyan{\lambda^\bfrak_4},\blue{\lambda^\dfrak_4}$ & $\cyan{\lambda^\bfrak_4}$ & $\azul{\lambda^\dfrak_4}$ \\

\hline

$3$ & $\green{\lambda^\bfrak_3}, \cyan{\lambda^\bfrak_4}, \blue{\lambda^\dfrak_4},\colive{\lambda^\dfrak_3}$ & $\green{\lambda^\bfrak_3}$ & $\colive{\lambda^\dfrak_3}$\\

\hline

$2$ & $\ccanela{\lambda^\bfrak_2},\green{\lambda^\bfrak_3}, \cyan{\lambda^\bfrak_4},\blue{\lambda^\dfrak_4}, \colive{\lambda^\dfrak_3}, \ccafe{\lambda^\dfrak_2}$ & $\ccanela{\lambda^\bfrak_2}$ & $\ccafe{\lambda^\dfrak_2}$\\

\hline

$1$ & $\cladri{\theta_1},\ccanela{\lambda^\bfrak_2},\green{\lambda^\bfrak_3}, \cyan{\lambda^\bfrak_4},\blue{\lambda^\dfrak_4}, \colive{\lambda^\dfrak_3}, \ccafe{\lambda^\dfrak_2},\corange{\lambda_1^\dfrak}$ & $\ccanela{\lambda^\bfrak_2}$ & $\cladri{\theta_1}$\\



\end{tabular}
\caption{Values of the cardinal characteristics of $S^{1,\dfrak}_i$ and Tukey connections from regular cardinals.}\label{tbl4.1}
\end{table}
\medskip

\noindent\textbf{Step~4.2} Construct a $(\lambda^\bfrak_1,\theta^-_1,\theta^-_1)$-directed system $\bar{N}^\bfrak_1:=\la N^\bfrak_{1,\alpha}:\, \alpha<\lambda^\bfrak_1\ra$ such that $N^\bfrak_{1,\alpha}\in N^\bfrak_{1,\alpha+1}$ (using $(\theta_1^-)^{<\theta^-_1}=\theta_1^-$). Thus $N^\bfrak_1$ is ${<}\lambda^\bfrak_1$-closed and $\Por^\bfrak_1:=\Por^\dfrak_1\cap N^\bfrak_1$ forces $\Rbf_i\eqT S_i^{1,\bfrak}:= S_i^{1,\dfrak}\cap N^\bfrak_1$. The values of the cardinals of $S_i^{1,\bfrak}$ are displayed in \autoref{tbl4.2}, in particular, we obtain $S_1^{1,\bfrak}\leqT\prod_{j=1}^4 \lambda^\dfrak_j\times\lambda^\bfrak_j$.

\begin{table}[ht]
\centering
\begin{tabular}{c||c|c|c}
     & (regular) & & \\
 $i$ & $\theta\leqT \Rbf_i$ & $\bfrak(S_i^{1,\bfrak})$ & $\dfrak(S_i^{1,\bfrak})$ \\
\hline\hline

$4$ & $\cyan{\lambda^\bfrak_4},\blue{\lambda^\dfrak_4}$ & $\cyan{\lambda^\bfrak_4}$ & $\azul{\lambda^\dfrak_4}$ \\

\hline

$3$ & $\green{\lambda^\bfrak_3}, \cyan{\lambda^\bfrak_4}, \blue{\lambda^\dfrak_4},\colive{\lambda^\dfrak_3}$ & $\green{\lambda^\bfrak_3}$ & $\colive{\lambda^\dfrak_3}$\\

\hline

$2$ & $\ccanela{\lambda^\bfrak_2},\green{\lambda^\bfrak_3}, \cyan{\lambda^\bfrak_4},\blue{\lambda^\dfrak_4}, \colive{\lambda^\dfrak_3}, \ccafe{\lambda^\dfrak_2}$ & $\ccanela{\lambda^\bfrak_2}$ & $\ccafe{\lambda^\dfrak_2}$\\

\hline

$1$ & $\cmelon{\lambda^\bfrak_1},\ccanela{\lambda^\bfrak_2},\green{\lambda^\bfrak_3}, \cyan{\lambda^\bfrak_4},\blue{\lambda^\dfrak_4}, \colive{\lambda^\dfrak_3}, \ccafe{\lambda^\dfrak_2},\corange{\lambda_1^\dfrak}$ & $\cmelon{\lambda^\bfrak_1}$ & $\corange{\lambda_1^\dfrak}$\\



\end{tabular}
\caption{Values of the cardinal characteristics of $S^{1,\bfrak}_i$ and Tukey connections from regular cardinals.}\label{tbl4.2}
\end{table}
\medskip

\noindent\textbf{Final step.} Let $N^\cfrak\preceq H_\chi$ be $\sigma$-closed such that $|N^\cfrak|=\lambda^\cfrak\subseteq N^\cfrak$ (using $(\lambda^\cfrak)^{\aleph_0}=\lambda^\cfrak$). Then $\Qor:=\Por^\bfrak_1\cap N^\cfrak$ forces $\Rbf_i\eqT S_i^{\cfrak}:= S_i^{1,\bfrak}\cap N^\cfrak$ and $\cfrak=\lambda^\cfrak$. By \autoref{fct:ScapN}~(d), the values in \autoref{tbl4.2} are still valid for $S_i^\cfrak$. Then $N:=N^\cfrak\cap\bigcap_{j=1}^4N^\dfrak_j\cap N^\bfrak_j$ is as desired, and $\Qor=\Por\cap N$.

\section{Discussion}

In our Cicho\'n's maximum result we get Tukey connections with products of ordinals, but it is unclear whether we actually have Tukey equivalence.

\begin{question}
Can we force Tukey equivalence in \autoref{thmcmax}~(a)?
\end{question}

Similarly, in \autoref{leftBCM}~(b), it is unclear whether we can force $\Rbf_4\eqT\lambda_5\times\lambda_4$.

Recall that, in \autoref{Cohenlimit}, we showed that the method of FS iterations restrict us to constellations of Cicho\'n's diagram where $\non(\Mwf)\leq\cov(\Mwf)$. There are four instances of Cicho\'n's maximum under this condition: the one proved in \autoref{thmcmax}, the one in \autoref{fig:cichon2} (after applying the same arguments in \autoref{sec:cmax} to the forcing from \autoref{leftKST}), and the two addressed in the following open question.

\begin{figure}[ht]
    \centering
    \includegraphics{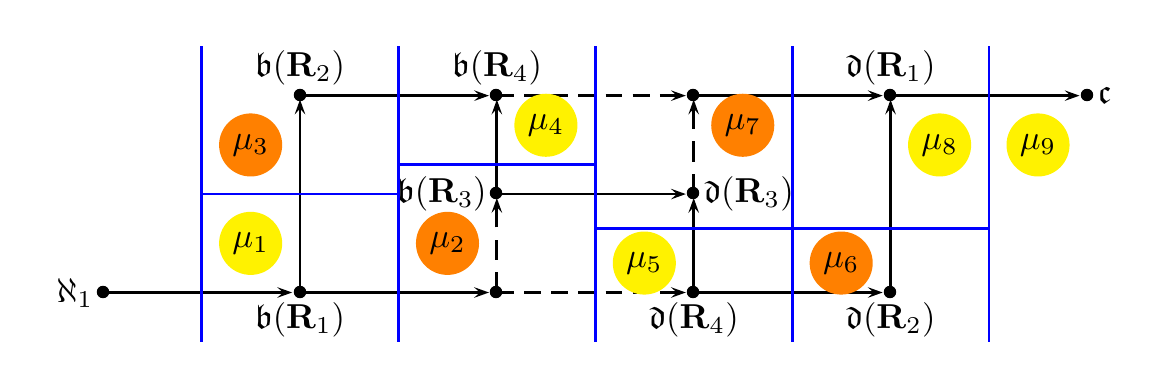}
    \caption{Another instance of Cicho\'n's maximum proved consistent with $\thzfc$. Here $\mu_i<\mu_j$ whenever $i<j$.}
    \label{fig:cichon2}
\end{figure}

\begin{question}
When $\mu_i<\mu_j$ for $i<j$, are the constellations of \autoref{fig:cichonQ} consistent with $\thzfc$?
\end{question}

\begin{figure}[ht]
    \centering
    \includegraphics{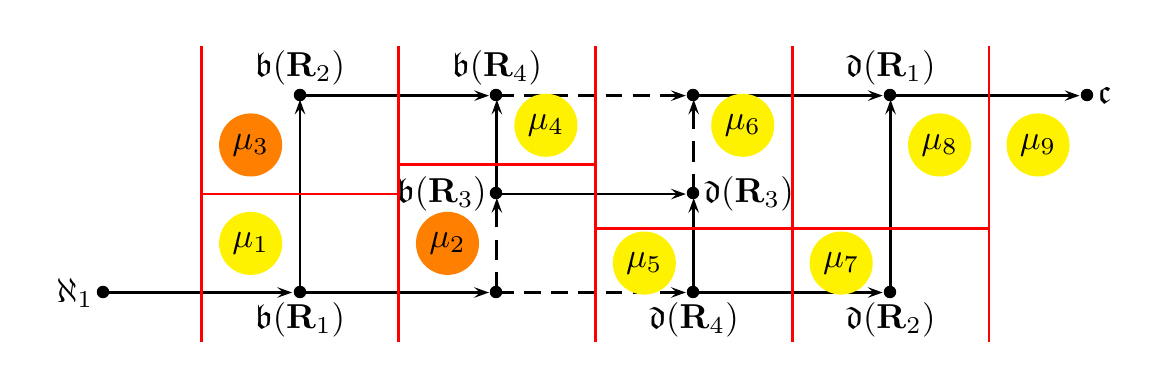}
    
    \includegraphics{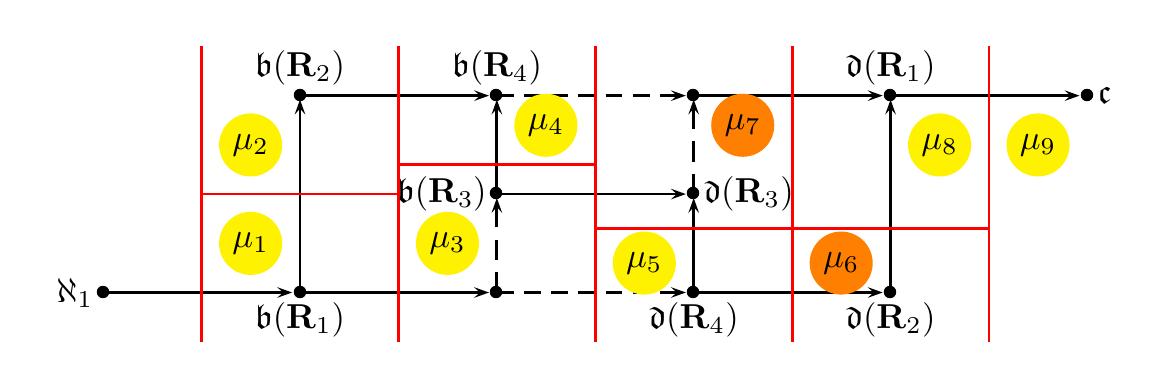}
    \caption{Instances of Cicho\'n's maximum in the context of $\non(\Mwf)\leq\cov(\Mwf)$ that have not been proved consistent with $\thzfc$.}
    \label{fig:cichonQ}
\end{figure}

On the other hand, no instance of Cicho\'n's maximum with $\cov(\Mwf)<\non(\Mwf)$ has been proved consistent so far.

We finish with some remarks about forcing singular values in Cicho\'n's diagram. In the models presented in this paper only $\cfrak$ can be singular, but there are some models with two singular values~\cite{Mvert}. There are also some instances of Cicho\'n's maximum with two singular values, but their consistency use large cardinals~\cite{E87}. The latter reference also presents interesting constellations in the random model.

Recently, Goldstern, Kellner, Shelah and the second author proved, using large cardinals, the consistency of Cicho\'n's maximum with the five cardinals on the right side possibly singular. Concretely, with the notation of \autoref{thmcmax}, it is forced $\Rbf_i\eqT[\lambda^\dfrak_i]^{<\lambda^\bfrak_i}$ for $1\leq i\leq 4$ by allowing $\lambda^\dfrak_i$ to be singular. However, it is still unknown how to adapt the methods of \autoref{sec:restrmod} and~\ref{sec:cmax} to prove this result (without using large cardinals).

\subsection*{Acknowledgments}
This paper was developed for the conference proceedings corresponding to the second virtual RIMS Set Theory Workshop ``Recent Developments in Set Theory of the Reals" that Professor Masaru Kada organized in October~2021. The authors are very thankful to Professor Kada for letting them participate in such wonderful workshop, and the second author is grateful for his invitation to talk about the methods presented in the second part of this paper.

The first author is supported by the Austrian Science Fund (FWF) P30666 and the DOC Fellowship of the Austrian Academy of Sciences at the Institute of Discrete Mathematics and Geometry, TU Wien;
the second author is supported by the Grant-in-Aid for Early Career Scientists 18K13448, Japan Society for the Promotion of Science.

This work is also supported by the Research Institute for Mathematical Sciences, an International Joint Usage/Research Center located in Kyoto University.

{\small
\bibliography{left}

\begin{thebibliography}{GKMS22}

\bibitem[Bar10]{BartInv}
Tomek Bartoszynski.
\newblock Invariants of measure and category.
\newblock In {\em Handbook of set theory. {V}ols. 1, 2, 3}, pages 491--555.
  Springer, Dordrecht, 2010.

\bibitem[BCM21]{BCM}
J\"{o}rg Brendle, Miguel~A. Cardona, and Diego~A. Mej\'ia.
\newblock Filter-linkedness and its effect on preservation of cardinal
  characteristics.
\newblock {\em Ann. Pure Appl. Logic}, 172(1):102856, 2021.

\bibitem[BJ95]{BJ}
Tomek {Bartoszy{\'n}ski} and Haim {Judah}.
\newblock {\em {Set Theory: On the Structure of the Real Line}}.
\newblock A K Peters, Wellesley, MA, 1995.

\bibitem[Bla10]{blass}
Andreas Blass.
\newblock Combinatorial cardinal characteristics of the continuum.
\newblock In {\em \-{} Handbook of set theory. {V}ols. 1, 2, 3}, pages
  395--489. Springer, Dordrecht, 2010.

\bibitem[Bre91]{Breciclar}
J\"{o}rg Brendle.
\newblock Larger cardinals in {C}icho\'{n}'s diagram.
\newblock {\em J. Symbolic Logic}, 56(3):795--810, 1991.

\bibitem[Bre22]{Bremodern}
J\"{o}rg Brendle.
\newblock Modern forcing techniques related to finite support iteration:
  Ultrapowers, templates, and submodels.
\newblock
  \href{https://arxiv.org/abs/2101.11494}{\texttt{arXiv:2101.11494[math.LO]}},
  2022.

\bibitem[CM19]{CM19}
Miguel~Antonio {Cardona} and Diego~Alejandro {Mej\'{i}a}.
\newblock {On Cardinal Characteristics of {Y}orioka Ideals}.
\newblock {\em Math. Log. Q.}, 65(2):170--199, 2019.

\bibitem[GKMS21]{GKMS}
Martin Goldstern, Jakob Kellner, Diego~Alejandro Mej\'{i}a, and Saharon Shelah.
\newblock Cicho\'n's maximum without large cardinals.
\newblock {\em J. Eur. Math. Soc. (JEMS)}, 2021.
\newblock online first,
  \href{https://ems.press/journals/jems/articles/3375044}{\texttt{doi:10.4171/JEMS/1178}}.

\bibitem[GKMS22]{E87}
Martin Goldstern, Jakob Kellner, Diego~Alejandro Mej\'{i}a, and Saharon Shelah.
\newblock Controlling classical cardinal characteristics while collapsing
  cardinals.
\newblock {\em Colloq. Math.}, 2022.
\newblock accepted,
  \href{https://arxiv.org/abs/1904.02617}{\texttt{arXiv:1904.02617[math.LO]}}.

\bibitem[GKS19]{GKScicmax}
Martin Goldstern, Jakob Kellner, and Saharon Shelah.
\newblock Cicho\'{n}'s maximum.
\newblock {\em Ann. of Math. (2)}, 190(1):113--143, 2019.

\bibitem[GMS16]{GMS}
Martin Goldstern, Diego~Alejandro Mej{\'{\i}}a, and Saharon Shelah.
\newblock The left side of {C}icho\'n's diagram.
\newblock {\em Proc. Amer. Math. Soc.}, 144(9):4025--4042, 2016.

\bibitem[JS90]{JSpres}
Haim Judah and Saharon Shelah.
\newblock The {K}unen-{M}iller chart ({L}ebesgue measure, the {B}aire property,
  {L}aver reals and preservation theorems for forcing).
\newblock {\em J. Symbolic Logic}, 55(3):909--927, 1990.

\bibitem[Kam89]{Ka}
Anastasis Kamburelis.
\newblock Iterations of {B}oolean algebras with measure.
\newblock {\em Arch. Math. Logic}, 29(1):21--28, 1989.

\bibitem[KST19]{KST}
Jakob Kellner, Saharon Shelah, and Anda~R. T\u{a}nasie.
\newblock Another ordering of the ten cardinal characteristics in
  {C}icho\'{n}'s diagram.
\newblock {\em Comment. Math. Univ. Carolin.}, 60(1):61--95, 2019.

\bibitem[Mej13]{M}
Diego~Alejandro Mej{\'{\i}}a.
\newblock Matrix iterations and {C}ichon's diagram.
\newblock {\em Arch. Math. Logic}, 52(3-4):261--278, 2013.

\bibitem[Mej19]{Mvert}
Diego~A. Mej{\'{i}}a.
\newblock Matrix iterations with vertical support restrictions.
\newblock In Byunghan Kim, J\"{o}rg Brendle, Gyesik Lee, Fenrong Liu,
  R~Ramanujam, Shashi~M Srivastava, Akito Tsuboi, and Liang Yu, editors, {\em
  Proceedings of the 14th and 15th Asian Logic Conferences}, pages 213--248.
  World Sci. Publ., 2019.

\bibitem[Mil81]{Mi}
Arnold~W. Miller.
\newblock Some properties of measure and category.
\newblock {\em Trans. Amer. Math. Soc.}, 266(1):93--114, 1981.

\bibitem[Voj93]{Vojtas}
Peter Vojt\'{a}\v{s}.
\newblock Generalized {G}alois-{T}ukey-connections between explicit relations
  on classical objects of real analysis.
\newblock In {\em Set theory of the reals ({R}amat {G}an, 1991)}, volume~6 of
  {\em Israel Math. Conf. Proc.}, pages 619--643. Bar-Ilan Univ., Ramat Gan,
  1993.

\end{thebibliography}
\bibliographystyle{alpha}
}

\Addresses

\end{document}